\documentclass{siamltex}

\usepackage[T1]{fontenc}
\usepackage{amsmath}
\usepackage{amssymb}
\usepackage{color}
\usepackage{graphicx}
\usepackage{algpseudocode} 
\usepackage{algorithm} 
\usepackage{lineno,hyperref}
\newcommand{\bi}{ \boldsymbol{i}}
\newcommand{\p}{ \mathcal{P} }
\newcommand{\A}{ \mathcal{A} }
\newcommand{\K}[1]{K_{#1}}
\newcommand{\M}[1]{M_{#1}}
\newcommand{\U}[1]{U_{#1}}
\newcommand{\D}[1]{D_{#1}}
\newcommand{\F}{ \boldsymbol{F} }
\renewcommand{\vec}{\mathrm{vec}}
\newcommand{\z}{\phantom{0}}

\newtheorem{thm}{Theorem}

\newtheorem{prop}{Proposition}

\begin{document}
\bibliographystyle{siam}

\pagestyle{myheadings}
\markboth{G. Sangalli and M. Tani}{Preconditioning for isogeometric problems}

\title {Isogeometric preconditioners based on fast solvers for the
  Sylvester equation\thanks{Version of February 22, 2016}}

\author{Giancarlo Sangalli$^\dagger$
\and  Mattia Tani\thanks{%
Universit\`a di Pavia, Dipartimento di Matematica ``F. Casorati'', 
Via A. Ferrata 1, 27100 Pavia, Italy. Emails: 
{\tt  \{sangia05,mattia.tani\}@unipv.it}}}
%
\maketitle
\begin{abstract} 
We consider large linear systems arising from the isogeometric
discretization of the Poisson problem on a single-patch domain. The numerical solution of such systems is considered a challenging task, particularly when the degree of the splines employed as basis functions is high.
We consider a preconditioning strategy which is based on the solution of a Sylvester-like equation at each step of an iterative solver. We show that this strategy, which fully exploits the tensor structure that underlies isogeometric problems, is robust with respect to both mesh size and spline degree, although it may suffer from the presence of complicated geometry or coefficients. 
We consider two popular solvers for the Sylvester equation, a direct
one and an iterative one, and we discuss in detail their
implementation and efficiency for 2D and 3D problems on single-patch
or conforming multi-patch  NURBS geometries.

Numerical experiments for problems with different domain geometries are presented, which demonstrate the potential of this approach.

\end{abstract}



\section{Introduction}

The Isogeometric method is a  computational technique  for solving
partial differential equations (PDEs). It has been proposed in  the  seminal paper 
\cite{Hughes2005} as an extension of the classical
finite element method, and is  based on the idea
of using splines or other functions constructed from splines (e.g.,
non-uniform rational B-splines, NURBS)
both for the parametrization  of  the computational domain, as it is
typically done by computer aided design software, and for the
representation of the unknown solution fields of  the PDE of interest.  
Many papers have demonstrated the effective advantage
of isogeometric methods  in various frameworks.  The interested reader
may find a detailed presentation of this idea with engineering  applications
in the book \cite{Cottrell2009}.
 

 Unlike standard finite element methods, the isogeometric method makes it possible to use
high-regularity functions. The so-called isogeometric $k$-method, 
based on splines of degree $p$ and global $C^{p-1}$
regularity, has shown significant  advantages in term of   higher accuracy per degree-of-freedom in comparison to $C^0$
finite elements of degree $p$ \cite{cottrell2007studies,acta-IGA}. However, the computational cost per
degree-of-freedom is also higher for the $k$-method, in currently
available  isogeometric codes. In practice, 
quadratic or cubic  splines are  typically preferred as they  maximise
computational efficiency. 

Standard  isogeometric codes  typically re-use  finite element technology, which is very
convenient but at the same time  not the best choice for computational
efficiency. The two fundamental
stages of a linear PDE solver are  the formation of the system matrix
$\A$ and the solution  of the linear system $\A u = b$, and   both
stages, in standard isogeometric software,  show a computational cost that grows
significantly as the degree $p$ grows. The focus of this paper is on
the second stage, that is the  linear solver.

The study of the computational efficiency of linear solvers for
isogeometric discretizations has been initiated in the papers \cite{Collier2012, Collier2013}, where it has been shown that the algorithms 
used  with the finite element method suffer of performance degradation 
when used to solve  isogeometric linear systems. Consider, for example, a Lagrangian finite element method with 
polynomial degree $p$ and $N$ degrees-of-freedom, in 3D: the
system matrix $\A$ has a storage cost of
$O(N p^3)$ non-zero terms and a solving cost by  a direct solver
 of  $O(N^2)$ floating point operations (FLOPs),
(see \cite[Section 2.3]{Collier2012}, under the assumption $N>p^6$). If, instead, we 
consider the isogeometric $k$-method with $C^{p-1}$  $p$-degree 
splines  and $N$ degrees-of-freedom, the system matrix
$\A$ has still $O(N p^3)$ non-zero entries,  but  a standard direct solver
costs $O(N^2 p^3)$ FLOPs, i.e., $p^3$ times  more than a finite
element approximation. 

Iterative solvers have attracted more attention in the isogeometric
community. The effort  has been primarily on the  development of
preconditioners for the  Poisson model problem, for
arbitrary degree and continuity splines.
 As reported in \cite{Collier2013},  standard algebraic preconditioners
 (Jacobi, SSOR, incomplete factorization) commonly adopted for finite
 elements exhibit reduced performance  when
used in the context of  the  isogeometric $k$-method.  Multilevel and multigrid approaches are studied respectively in \cite{Buffa2013} and \cite{Gahalaut2013}, while advances in the theory
of   domain-decomposition based
solvers are given in,  e.g., \cite{Veiga2012,BeiraodaVeiga2013,bercovier2015overlapping}.
These papers also confirm the difficulty in achieving both \emph{robustness}
and computational \emph{efficiency} for the high-degree $k$-method. 
 In this context, we say that a preconditioner $\p$ for
the linear system  $\A u = b$ is robust if the condition number $\kappa\left(\p^{-1} \A \right) $
is bounded from above by a reasonably low number, independent of the degree or continuity of the spline space
adopted in the isogeometric discretization;  we say that a
preconditioner  is computationally  efficient if its setup and
application has a computational cost comparable to the one of the
matrix-vector product for the system matrix $\A$, i.e. $O(N
p^3)$.

More sophisticated  multigrid preconditioners have been proposed in
the recent papers \cite{Donatelli2015} and \cite{Hofreither2015}. The
latter, in particular, contains a proof of robustness, based on the
theory of \cite{Takacs2015}. 
The two works ground on the following common ingredients:  specific spectral properties of the discrete operator of the isogeometric
$k$-method and  the tensor-product structure of
isogeometric spaces.

In this paper we also exploit  the  tensor-product structure of
multivariate spline space, on a different basis. We rely on  approaches that have been developed
for the so-called Sylvester equation.

Consider the Laplace
operator with constant coefficients, on the square $[0,1]^2$, then  the tensor-product spline
Galerkin discretization leads to the system
\begin{equation}
  \label{eq:kron2dintro}
  ( \K1 \otimes \M2 + \M1 \otimes \K2) u = b,
\end{equation}
where $\K{\ell}$ and $\M{\ell}$ denote the univariate stiffness and
mass matrices in the $\ell$ direction, $\ell = 1,2$, and $\otimes$ is the Kronecker
product. Equation \eqref{eq:kron2dintro}, when reformulated as a matrix equation, 
takes the name of (generalized) Sylvester equation.
This is a well
studied problem in the numerical linear algebra literature, as it appears in
many applications, e.g. stochastic
PDEs,  control theory, etc. (see the recent survey \cite{Simoncini2013}). 
Observe that in general, for variable
coefficients, general elliptic problems, non trivial and possibly multi-patch geometry parametrization, the isogeometric system is not as in
\eqref{eq:kron2dintro}. In this case, a fast solver for
\eqref{eq:kron2dintro}  plays the role of a  preconditioner. Having this
motivation in mind, our aim is to study how the linear solvers for
the Sylvester equation perform for \eqref{eq:kron2dintro}, especially when originated by 
an  isogeometric $k$-method.

We select two among the most popular algorithms: the first is the 
fast diagonalization (FD) direct solver proposed by Lynch, Rice and Thomas in \cite{Lynch1964}, the second is
the alternating direction implicit (ADI) iterative  solver, first 
introduced in \cite{Peaceman1955} and further developed in a number of papers, among which \cite{Wachspress1963}.  The potential of 
ADI in the context of isogeometric problems has already been recognized  in \cite{Gao2013}. 

Our ultimate goal is the solution of 3D isogeometric systems, especially when high
resolution is needed. A remarkable example is the simulation of
turbulence, see e.g. \cite{bazilevs2010isogeometric}. 
Here the Poisson problem on a unit cube leads to the linear system
$$ ( \K1 \otimes \M2 \otimes \M3 + \M1 \otimes \K2 \otimes \M3 + \M1 \otimes \M2 \otimes \K3) u = b, $$
which can be efficiently solved with the FD method and a generalization of ADI.

We analyze and benchmark the proposed approaches in both 2D and 3D cases.
The results show that the FD method exceeds by far ADI in terms of computational efficiency. 
This is seen in 2D, but the gap is much wider in 3D.
Here, the FD solver count is $O(N^{4/3})$ FLOPs.
 ADI costs  $O(Np)$ FLOPs per iteration, which result in an asymptotically
lower operation counting. However, when used as a preconditioner
in benchmarks that are representative of realistic problems, the FD
solver performs orders of magnitude better. In fact, in all our
benchmarks that uses a conjugate
gradient (CG) iterative solver, the computational time spent in the
 FD preconditioner application is even lower  than the
 residual computation (multiplication of matrix $\A$ times a vector).
This surprising performance is due to the fact that the
FD solver requires dense matrix-matrix operations that
takes advantages of modern computer architecture. In particular, the
performance boost is due to the efficient use of the CPU hierarchy
cache and  memory access.
Furthermore, the FD  method is especially suited to parallelization which 
may  significantly speed up the execution time, though
this is not considered in our analysis.

Concerning robustness, in both approaches the
condition number $\kappa\left(\p^{-1} \A \right) $ depends only on the
geometry parametrization of the computational domain $\Omega$ and on
the coefficients of the differential operator. In this paper we 
study this dependence  and perform some numerical tests. We will show
that a singular mapping causes a loss of robustness, while
$\kappa\left(\p^{-1} \A \right) $ is uniformly bounded w.r.t.
the degree $p$ and mesh size  $h$ when the parametrization is
regular. In all cases, it is important to have strategies to further improve
the condition number, but this goes beyond the scope of the present
work and this will be the topic of further researches.

Finally, we show how to combine the considered preconditioners with a domain decomposition approach in order to solve multi-patch problems with conforming discretization. The overall strategy naturally inherits efficiency and robustness from the preconditioners for single-patch problems.

The outline of the paper is as follows. In Section \ref{sec:preliminaries} we 
introduce the matrices stemming from isogeometric discretization; we also 
recall the Kronecker product notation and its main properties. In
Section \ref{sec:preconditioner} we define the preconditioner and
discuss the spectral condition number of the preconditioned
system. Sections \ref{sec:2D} and \ref{sec:3D} describe how such
preconditioner can be efficiently applied in the 2D and 3D case,
respectively. 
In Section \ref{sec:multi-patch}, the multi-patch setting is discussed.
Numerical experiments are reported in Section \ref{sec:experiments}. Finally, in Section \ref{sec:conclusions} we draw the conclusions and outline future research directions.

\section{Preliminaries} \label{sec:preliminaries}

\subsection{Splines-based  isogeometric method}
\label{sec:splin-isog-meth}

We consider, as a model problem, the Poisson equation with Dirichlet
boundary conditions:
\begin{equation}\label{poisson}
\left\{
\begin{aligned}
-\text{div} ( {K} (\mathbf{x}) \nabla \mathrm{u} (\mathbf{x}) ) = f (\mathbf{x}) &\quad \text{on} \quad \Omega, \\
\mathrm{u} = 0 & \quad \text{on} \quad \partial \Omega,
\end{aligned}
\right.
\end{equation}
where $ {K}  (\mathbf{x})$ is a  symmetric positive definite matrix for each $ \mathbf{x}\in \Omega$. 
In isogeometric methods, $\Omega$ is given by a spline or NURBS
parametrization. For the sake of simplicity, we consider a
single-patch spline parametrization.

Given two positive integers $p$ and $m$, we say that $\Xi :=
\{\xi_1,\dots, \xi_{m+p+1}\}$ is an open knot vector if
\begin{equation}
  \label{eq:knot-vector}
  \xi_1 =\ldots=\xi_{p+1} < \xi_{p+2} \le \ldots \le 
\xi_{m} < \xi_{m+1}=\ldots=\xi_{m+p+1},
\end{equation}
where repeated knots are allowed, up to multiplicity $p$. Without loss of generality, we
assume $\xi_1 = 0$ and $\xi_{m+p+1} =1 $.
From the knot vector $\Xi$, B-spline functions of degree $p$ are
defined e.g. by the Cox-DeBoor recursive formula:  piecewise constants
($p=0$) B-splines are
\begin{equation}\label{eq:cox-deboor-1}
\hat B_{i,0}(\zeta) = \left \{
\begin{array}{ll}
1 & \text{if } \xi_i \leq \zeta < \xi_{i+1}, \\
0 & \text{otherwise},
\end{array}
\right.
\end{equation}
and for $p \ge 1$ the \emph{B-spline}  functions are obtained  by the recursion
\begin{equation}\label{eq:cox-deboor-2}
\hat B_{i,p}(\zeta) = \frac{\zeta - \xi_i}{\xi_{i+p} - \xi_i} \hat B_{i,p-1}(\zeta) + \frac{\xi_{i+p+1} - \zeta}{\xi_{i+p+1} - \xi_{i+1}} \hat B_{i+1,p-1}(\zeta),
\end{equation}
where $0/0 = 0$. In general,  the  B-spline functions  are  degree $p$ piecewise polynomial
with $p-r$ continuous derivative  at each knot with multiplicity
$r$. 
In this work we are primarily  interested in the  so called
$k$-refinement or isogeometric $k$-method,
see \cite{Cottrell2009}. For that, we assume that the multiplicity of all internal knots
is $1$, which corresponds to $C^{p-1}$ continuous splines.
Each B-spline $\hat B_{i,p}$ depends only on $p+2$ knots, which are collected in the \emph{local knot vector}
$$\Xi_{i,p}:= \{ \xi_{i}, \ldots,\, \xi_{i+p+1}\}.$$
When needed, we will  adopt the notation $  \hat
B_{i,p}(\zeta) = \hat B[\Xi_{i,p}] (\zeta).$ 
The support of each basis function is exactly ${\rm supp}(\hat B_{i,p}) = [\xi_i , \xi_{i+p+1}]$.  

Multivariate B-splines in dimension $d$  ($d=2,3$ are the interesting
cases) are defined from univariate B-splines by
tensorization. For the sake of simplicity, we assume that the degree $p$ and the length of the knot
vectors $\Xi^\ell = \{\xi_{\ell,1}, \ldots, \xi_{\ell,m + p + 1} \}$, is the same
in all directions $\ell =1,\ldots,d$. Then  for each multi-index  $\bi=(i_1,\ldots, i_d)$, we
introduce   the local knot vector $ \Xi_{i_1,p}^1 \times
\ldots\times \Xi_{i_d,p}^d $ and the multivariate B-spline
\begin{equation}
  \label{eq:multivariate-B-splines}
  \hat B_{\bi,p}(\boldsymbol \zeta) = \hat
    B[\Xi_{i_1,p}^1](\zeta_1) \ldots \hat
    B[\Xi_{i_d,p}^d](\zeta_d).
\end{equation}
To simplify the notation, when not needed the subscript $p$ is not
indicated. 

The domain $\Omega$ is
given by a $d$-dimensional single-patch spline parametrization\begin{displaymath}
  \Omega = \boldsymbol{F} ([0,1]^d), \text{ with }
  \boldsymbol{F}(\boldsymbol \zeta) = \sum_{\bi}
  \boldsymbol{C}_{\bi} \hat  B_{\bi} (\boldsymbol \zeta),
\end{displaymath}
where $  \boldsymbol{C}_{\bi} $ are the control points. Following the 
isoparametric paradigm, the basis functions $B_{\bi}$ on  $\Omega$
are defined as $B_{\bi} = \hat B_{\bi}\circ \boldsymbol{F} ^{-1}$.  
The isogeometric space, incorporating the homogeneous Dirichlet
boundary condition, reads
\begin{equation}\label{eq:Vh}
  V_h = \text{span}\{B_{\bi} \text{ such that }  \bi=(i_1,\ldots,
  i_d), \text{ with } 2\leq i_{\ell}\leq m-1, 1\leq \ell \leq d \}.
\end{equation}
We introduce a scalar indexing for functions in \eqref{eq:Vh} as
follows: to the each multi-index $\bi=(i_1,\ldots, i_d)$ we associate $i =
1+ \sum_{\ell = 1}^{d} (m-2)^{\ell-1}(i_{\ell}-2)$ and, with  abuse
of notation,  indicate $
B_{\bi}= B_{i}$, etc. The dimension of $V_h$ is denoted as $N=n^d$,
where $n = m-2$. Then, the Galerkin  stiffness matrix  reads 
\begin{equation}
\label{eq:A}
  \begin{aligned}
    \A_{ij} & =   \int_{\Omega} \left(\nabla B_{i}(\boldsymbol{x}) \right)^T {K} (\boldsymbol{x})
  \nabla B_{{j}}(\boldsymbol{x}) d\boldsymbol{x}\\
& =   \int_{[0,1]^d} \left(\nabla \hat{B}_i(\boldsymbol{\xi})\right)^T
Q(\boldsymbol{\xi}) \,  \nabla
\hat{B}_j \left(\boldsymbol{\xi}\right)\; d\boldsymbol{\xi}, \qquad i,j = 1,\ldots,N  
  \end{aligned}
\end{equation}
where
\begin{equation} \label{eq:Q} Q = \mbox{det}\left(J_{\F}\right)  J_{\F}^{-T} K J_{\F}^{-1} \end{equation}
and $J_{\F}$ denotes the Jacobian of $\F$. 

 The support of a B-spline in $V_h$  that does not touch $\partial
  \Omega$ intersects the support of $(2p+1)^d $ splines in $V_h$
  (including itself).  If the support of a   B-splines intersects with $\partial
  \Omega$, it overlaps at least   $(p+1)^d $ and up to  $(2p+1)^d $
  B-spline supports (including itself). 
Thus, the number of nonzeros of $\A$ is about $ (2p+1)^d N$. 

\subsection{Kronecker product}


Let $A \in \mathbb{R}^{n_a \times n_a}$, and $B \in \mathbb{R}^{n_b \times n_b}$. The \textit{Kronecker product} between $A$ and $B$ is defined as
$$ A \otimes B = \begin{bmatrix} a_{11} B & \ldots & a_{1 n_a} B \\ \vdots & \ddots & \vdots \\ a_{n_a 1} B & \ldots & a_{n_a n_a} B \end{bmatrix} \; \in \mathbb{R}^{n_a n_b \times n_a n_b}, $$
where $a_{ij}$, $i,j=1,\ldots n_a$, denote the entries of $A$. The Kronecker product is an associative operation, 
and it is bilinear with respect to matrix sum and scalar multiplication.
We now list a few properties of the Kronecker product that will be useful in the following.

\begin{itemize}
\item It holds \begin{equation} \label{eq:krontranspose} (A \otimes B)^T = A^T \otimes B^T. \end{equation}
\item If $C$ and $D$ are matrices of conforming order, then 
\begin{equation} \label{eq:mixedproduct} \left(A \otimes B\right) \left(C \otimes D\right) = (AC \otimes BD ). \end{equation}
\item If $A$ and $B$ are nonsingular, then 
\begin{equation} \label{eq:kroninv} \left(A \otimes B\right)^{-1} = A^{-1} \otimes B^{-1}. \end{equation}
\item If $\lambda_i$, $i=1,\ldots,n_a$, denote the eigenvalues of $A$ and $\mu_b$, $j=1,\ldots,n_2$, denote the eigenvalues of $B$, then the $n_a n_b$ eigenvalues of $A \otimes B$ have the form 
\begin{equation} \label{eq:kroneig} \lambda_i \mu_j, \qquad i=1,\ldots,n_a, \; j=1,\ldots,n_b. \end{equation}
\end{itemize}

Property \eqref{eq:krontranspose} implies that if $A$ and $B$ are both symmetric, then $A \otimes B$ is also symmetric. Moreover, if $A$ and $B$ are both positive definite, then according to \eqref{eq:kroneig} $A \otimes B$ is also positive definite.

For any matrix $ X \in \mathbb{R}^{n_a\times n_b}$ we denote with $ \vec(X)$ the vector of $\mathbb{R}^{n_a n_b}$ obtained by ``stacking'' the columns of $X$. Then if $A$, $B$ and $X$ are matrices of conforming order, and $x = \vec(X)$, it holds
\begin{equation} \label{eq:kronprod} \left( A \otimes B\right)x = \vec( B X A ^T). \end{equation}
This property can be used to cheaply compute matrix-vector products
with a matrix having Kronecker structure. Indeed, it shows that
computing $\left( A \otimes B\right)x$ is equivalent to computing $ n_b$ matrix-vector products with $A$ and $n_a$ matrix-vector products with $B$. Note in particular that $\left( A \otimes B\right)$ does not have to be formed. 

If $A$ and $B$ are nonsingular, then \eqref{eq:kronprod} is equivalent to
\begin{equation} \label{eq:kronsolve} \left( A \otimes B\right)^{-1}x = \vec( A^{-1} X B^{-T} ), \end{equation}
which, in a similar way, shows that the problem of solving a linear
system having $\left( A \otimes B\right)$ as coefficient matrix is
equivalent to solve $ n_b$ linear systems involving $A$ and $ n_a $ linear systems involving $B$.

\subsection{Evaluation of the computational cost and efficiency}

Throughout the paper, we will primarily evaluate the computational cost of an
algorithm by counting the number of floating point operations (FLOPs)
it requires.  
A single addition, subtraction, multiplication or division performed in floating point arithmetic counts as one FLOP \cite{Golub2012}. 
The number of FLOPs associated with an algorithm is an
indication to  assess its efficiency, and it is widely employed in literature.
However, any comparison of FLOPs between different algorithms
should be interpreted with caution.  We emphasize, indeed, that the
number of FLOPs represents just a portion of the computational effort
required by an algorithm, as it does not take into account the 
movement of data in the memory and other overheads that affect the
execution time. While these are  difficult to estimate, we
will discuss them when needed.
  
\section{The preconditioner} \label{sec:preconditioner}

Consider the matrix
\begin{equation} \label{eq:prec} \p_{ij} = \int_{[0,1]^d} \left(\nabla \hat{B}_i\right)^T \nabla \hat{B}_j \; d \xi, \qquad i,j = 1,\ldots,N. \end{equation}
Observe that $\p=\A$ in the special case when  $K$ and $J_{\F}$ are the identity matrices, which means in particular that $\Omega$ is the unit cube. 
For $d=2$, by exploiting the tensor product structure of the basis functions we have
$$ \p = \K1 \otimes \M2 + \M1 \otimes \K2, $$
where $\M1,\M2$ represent the mass,  and $\K1,\K2$  the stiffness
univariate matrices. 
$$ \left( \M1 \right)_{ij} = \int_0^1 \hat B[\Xi_{i}^1](\zeta_1) \, \hat B[\Xi_{j}^1](\zeta_1) \; d  \zeta_1, \qquad \left( \M2 \right)_{ij} = \int_0^1 \hat B[\Xi_{i}^2](\zeta_2) \, \hat B[\Xi_{j}^2](\zeta_2) \;  d  \zeta_2, $$
$$ \left( \K1 \right)_{ij} = \int_0^1  (\hat B[\Xi_{i}^1])'(\zeta_1) \cdot  (\hat B[\Xi_{j}^1])'(\zeta_1) \;  d  \zeta_1, $$
$$ \left( \K2 \right)_{ij} = \int_0^1 (\hat B[\Xi_{i}^2])'(\zeta_2) \cdot  (\hat B[\Xi_{j}^2])'(\zeta_2) \;  d  \zeta_2. $$
Such matrices are all symmetric positive definite and banded with
bandwidth $p$ (we say that a matrix $B$ has bandwidth $p$ if $B_{ij}
= 0$ for $\left|i-j\right| > p$). These matrices  have the same order
$n=m-2$, where $m$ is the dimension of the univariate spline space
(see also  from \eqref{eq:knot-vector}). Similarly, when $d=3$
$$ \p = \K1 \otimes \M2 \otimes \M3 + \M1 \otimes \K2 \otimes \M3 + \M1 \otimes \M2 \otimes \K3. $$
By comparing \eqref{eq:A} and \eqref{eq:prec}, observe that $\p_{ij} \neq 0$ if and only if $\A_{ij} \neq 0$. Thus, despite having different entries in general, $\A$ and $\p$ have the same sparsity pattern.

We propose $\p$, defined in \eqref{eq:prec}, as a preconditioner for
the isogeometric matrix $\A$. 
In other words, we want to precondition a problem with arbitrary
geometry and coefficients with a solver for the same operator on the
parameter domain, with constant coefficients. This is a common
approach, see e.g. \cite{Gao2013}, \cite{Donatelli2015} and
\cite{Hofreither2015}.

Note that, according to \eqref{eq:krontranspose} and \eqref{eq:kroneig}, $\p$ is symmetric and positive definite,
and hence we can use it as preconditioner for the conjugate gradient (CG)
method. At each CG iteration, we need to solve a system of the form
\begin{equation} \label{eq:preceq} \p s = r, \end{equation}
where $r$ is the current residual.  Due to the structure of $\p$,
\eqref{eq:preceq} is a Sylvester-like equation.  How to efficiently solve this
system for $d=2$ and $d=3$, employing solvers for Sylvester
equation, will be the topic of Sections \ref{sec:2D} and \ref{sec:3D}.
In this section we discuss the effects of geometry and coefficients on
the overall CG convergence. 
The next proposition provides an upper bound for the spectral
condition number of $\p^{-1} \A $. 

\begin{prop} \label{prop:bound1}

It holds
\begin{equation} \label{eq:bound} \kappa\left(\p^{-1} \A \right) \leq \frac{ \displaystyle \sup_{\Omega} \lambda_{\max}\left(Q\right)}{ \displaystyle \inf_{\Omega} \lambda_{\min}\left(Q\right)}, 
\end{equation}
%
where the  matrix $Q$ is given in \eqref{eq:Q}.

\end{prop}

\begin{proof} 

Let $u = (u_1, \ldots, u_N)^T \in \mathbb{R}^n$, and define $u_h = \sum_{i=1}^{N} u_i \hat{B}_i$. Then it holds
\begin{eqnarray*} u^T \A u & = &  \int_{[0,1]^d} \left(\nabla u_h \right)^T Q \; \nabla u_h  \; d \xi \leq \int_{[0,1]^d} \lambda_{\max}\left(Q\right) \left\| \nabla u_h \right\|^2 \; d \xi \\[5pt]
& \leq & \sup_{\Omega} \lambda_{\max}\left(Q\right) \int_{[0,1]^d} \left\| \nabla u_h \right\|^2 \; d \xi = \sup_{\Omega} \lambda_{\max}\left(Q\right) \; u^T \p u. \end{eqnarray*} 
By the Courant-Fischer theorem, we infer $\lambda_{\max}\left(\p^{-1} \A \right) \leq \displaystyle \sup_{\Omega} \lambda_{\max}\left(Q\right)$. With analogous calculations one can show that $\lambda_{\min}\left(\p^{-1} \A \right) \geq \displaystyle \inf_{\Omega} \lambda_{\min}\left(Q(\xi)\right)$, and hence $\kappa\left(\p^{-1} \A \right) \leq  \frac{\displaystyle\sup_{\Omega} \lambda_{\max}\left(Q\right)}{\displaystyle \inf_{\Omega} \lambda_{\min}\left(Q\right)} $.
\end{proof}

Proposition \ref{prop:bound1} states a useful and well-known result, that
formalizes an intuitive fact. As long as the considered problem does not depart much
from the model problem on the square with constant coefficients, the
right-hand side of \eqref{eq:bound} will be small and the
preconditioner is expected to perform well. On the other hand, if the
eigenvalues of $Q$ vary widely, due to the presence of complicated
geometry or coefficients, the preconditioner performance
decreases. In these cases, it is useful to have  strategies to improve
the spectral conditioning of $\p^{-1} \A$: this is a topic that we
will address in a forthcoming paper.
We emphasize that bound \eqref{eq:bound} does not depend neither on the mesh size nor on the spline degree, but only on ${\F}$ and $K$.

Furthermore Proposition \ref{prop:bound1} allows one to compare the strategies proposed in this paper with different approaches which do not rely on the preconditioner $\p$. This might be helpful not only from the theoretical point of view, but also from a practical perspective. Indeed, during the process of assembling the stiffness matrix $\A$, the matrix
$Q$ can be  evaluated at all quadrature points of the mesh and  the extreme eigenvalues of $Q$  can be computed
in order to estimate the right-hand side of \eqref{eq:bound}. This
leads to  a reliable estimate of $\kappa\left(\p^{-1} \A \right)$ before attempting to solve the system. 
If another solver is available, which does not suffer from complicated
geometry or coefficients, then a smart software could use the estimate
on $\kappa\left(\p^{-1} \A \right)$ to automatically choose which
method is more suited to solve the system at hand. 

\section{The 2D case} \label{sec:2D}

When $d=2$, equation \eqref{eq:preceq} takes the form 
\begin{equation} \label{eq:kron2d} \left(\K1 \otimes \M2 + \M1 \otimes \K2 \right) s = r. \end{equation}
Using relation \eqref{eq:kronprod}, we can rewrite this equation in matrix form
\begin{equation} \label{eq:sylvester} \M2 S \K1 + \K2 S \M1 = R, \end{equation}
where $\vec(S) = s$ and $\vec(R) = r$. Equation \eqref{eq:sylvester} takes the name of (generalized) Sylvester equation.
Due to its many applications, the literature dealing with Sylvester equation (and its variants) is vast, and a number of methods have been proposed for its numerical solution. We refer to \cite{Simoncini2013} for a recent survey on this subject.

In the last two decades, the research on Sylvester equation has mainly focused on methods 
which require that the right-hand side matrix $R$ has low rank.
Such methods are nor considered in this work.
Indeed, even if there are cases where  $R$ is low-rank or can be
approximated efficiently by a low-rank matrix, this is not
the general case. Furthermore, and perhaps more
important, low-rank methods are  designed for solving very large
problems, where even storing the solution $S$ might be
unfeasible. This is not the case of
problems of practical interest in isogeometric analysis.

In this paper, we consider two among the most studied methods, which in the authors' perspective seem the most suited for the particular features of IGA problems. 
The fast diagonalization (FD) method is a direct solver,  which means that $s =
\p^{-1}r$ is computed exactly. The alternating direction implicit
(ADI) method  is an iterative solver, which means that $s$ is computed only
approximately. We remark that ADI was first applied to IGA problems in \cite{Gao2013}.

To keep the notation consistent with the rest of the paper, in this
section we will favor the Kronecker formulation \eqref{eq:kron2d} with
respect to the matrix equation form \eqref{eq:sylvester}.

\subsection{The 2D fast diagonalization method} \label{sec:direct2d}

We describe a direct method for \eqref{eq:kron2d} that was first presented in 1964 by Lynch, Rice and Thomas \cite{Lynch1964} as a method for solving elliptic partial differential equations discretized with finite differences. Following \cite{Deville2002}, we refer to it as the fast diagonalization (FD) method.
We remark that this approach was extended to a general Sylvester equation involving nonsymmetric matrices by Bartels and Stewart in 1972 \cite{Bartels1972}, although this is not considered here. 


We consider the generalized eigendecomposition of the matrix pencils $(\K1,\M1)$ and $(\K2,\M2)$, namely 
\begin{equation} \label{eq:eigen2d} \K1 \U1 = \M1 \U1 \D1 \qquad \K2 \U2 = \M2 \U2 \D2, \end{equation}
where $\D1$ and $\D2$ are diagonal matrices whose entries are the eigenvalues of $\M1^{-1}\K1$ and $\M2^{-1}\K2$, respectively, while $\U1$ and $\U2$ satisfy
%
$$ \U1^T \M1 \U1 = I, \qquad \U2^T \M2 \U2 = I, $$
%
which implies in particular $\U1^{-T} \U1^{-1} = \M1$ and $\U2^{-T}
\U2^{-1} = \M2$, and also,  from \eqref{eq:eigen2d},  $\U1^{-T}
\D1  \U1^{-1} = \K1$ and $\U2^{-T} \D2
\U2^{-1} = \K2$. Therefore we  factorize $\p$ in \eqref{eq:kron2d} as follows:
$$ \left(\U1 \otimes \U2 \right)^{-T} \left(\D1 \otimes I + I
  \otimes \D2 \right)\left(\U1 \otimes \U2\right)^{ -1} s = r, $$ 
and adopt the following strategy:

\begin{algorithm}
\caption{FD direct method (2D)}\label{FD}
\begin{algorithmic}[1]
\State Compute the generalized eigendecompositions \eqref{eq:eigen2d}
\State Compute $\tilde{r} = (\U1 \otimes \U2) ^{ T} r $
\State Compute $\tilde{s} = \left(\D1 \otimes I + I \otimes \D2 \right)^{-1} \tilde{r} $
\State Compute $s = (\U1 \otimes \U2) \tilde{s} $
\end{algorithmic}
\end{algorithm}

\textbf{Computational cost.} The exact cost of the eigendecompositions in
line 1  depends on
the algorithm employed.  
We refer to \cite[Chapter 8]{Golub2012}, \cite[Section 5.3]{Demmel1997} and references therein for an overview of the state-of-the-art methods.
A simple approach is to first compute the Cholesky factorization $\M1
= L L^T $ and the symmetric matrix $\widetilde{\K1} = L^{-1} \K1
L^{-T}$. Since $\M1$ and $\K1$ are banded, the cost of these
computations is O($p n^2$) FLOPs. The eigenvalues of $\widetilde{\K1}$
are the same of \eqref{eq:eigen2d}, and once the matrix
$\widetilde{\U1}$ of orthonormal eigenvectors is computed then one can compute $\U1 =
L^{-T} \widetilde{\U1} $, again at the cost of O($p n^2$) FLOPs. Being $\widetilde{\U1}$  orthogonal, then $\U1^T \M1 \U1 =
I_n$.
If the eigendecomposition of $\widetilde{\K1}$ is computed using a divide-and-conquer method, the cost of this operation is roughly $4 n^3$ FLOPs. We remark that the divide-and-conquer approach is also very suited for parallelization.
In conclusion, by this approach, line 1 requires roughly $8 n^3$ FLOPs. 

Lines 2 and 4 each involve a matrix-vector product with a matrix
having Kronecker structure, and each step is equivalent (see \eqref{eq:kronprod}) to
$2$ matrix products involving dense $n \times n$ matrices. The total
computational cost
of both steps is $8n^3$ FLOPs. Line 3 is just a diagonal scaling, and
its $O(n^2)$ cost is negligible.  
We emphasize that the overall computational cost of Algorithm \ref{FD} is independent of $p$. 

If we  apply Algorithm \ref{FD} as a preconditioner, then Step 1 may
be performed only once, since the matrices involved do not change
throughout the CG iteration. In this case the main cost can be quantified
in approximately  $ 8 n^3$ FLOPs per CG iteration. The other main computational effort of each CG
iteration is the residual computation, that is the  product of the system matrix $\A $ by a vector,
whose cost in FLOPs is twice the number of nonzero entries of  $\A $, that is approximately
$ 2(2p+1)^2 n^2$. In conclusion, the  cost ratio between the preconditioner
application and the residual computation  is about $4n/(2p+1)^2\approx n/p^2$. 

\subsection{The ADI method}

The ADI method was originally proposed in 1955 by Peaceman and
Racheford as a method to solve elliptic and parabolic differential
equations with two space variables \cite{Peaceman1955} by a
structured finite difference discretization. 
An important contribution to the early development of ADI is due to Wachspress and collaborators \cite{Wachspress1962} \cite{Wachspress1963} \cite{Lu1991} \cite{Ellner1991}. 
A low-rank version of ADI, which is not considered here, was first proposed in 2000 by Penzl \cite{Penzl2000}, 
and since then has become a very popular approach, see e.g. \cite{Damm2008} \cite{Benner2009} \cite{Benner2013} \cite{Kressner2014}.
For more details on the classical ADI method, we refer to the monograph \cite{Wachspress2013}.

The $j-$th iteration of the ADI method applied to equation \eqref{eq:kron2d} reads:
\begin{subequations} \label{eq:ADI}
\begin{eqnarray} 
\left(\left(\K1 + \omega_j \M1 \right) \otimes \M2 \right)  s_{j - 1/2} & = & r - \left( \M1 \otimes \left(\K2 - \omega_k \M2 \right) \right) s_{j-1}, \\
 \left( \M1 \otimes \left(\K2 + \gamma_j \M2 \right) \right) s_{j} & = & r - \left( \left(\K1 - \gamma_j \M1 \right) \otimes \M2  \right) s_{j-1/2}, 
\end{eqnarray} 
\end{subequations}
where $\gamma_j,\omega_j \in \mathbb{R}$ are acceleration parameters, which will be discussed in the next section. 
Note that at each ADI iteration we need to solve two linear systems
where the coefficient matrix has a Kronecker product structure, and this
can be done efficiently by \eqref{eq:kroninv}.

\subsubsection{Convergence analysis}

Let $J $ denote the total number of ADI iterations, and let $e_J = s - s_J$, where $s$ is the exact solution of \eqref{eq:kron2d}, denote the final error. Then it holds
\begin{equation} \label{eq:ADIerr} e_{J} = M^{-1/2} T_J M^{1/2} e_0, \end{equation}
where
\begin{equation} \label{eq:Tj} T_J = \prod_{j = 1}^{J} 
  \left(\widetilde{\K1} - \gamma_j I \right)  \left(\widetilde{\K1} +
    \omega_j I  \right)^{-1} \otimes \left(\widetilde{\K2} + \gamma_j I  \right)^{-1}\left(\widetilde{\K2} - \omega_j I \right), \end{equation}
with $M = \M1 \otimes \M2 $, $\widetilde{\K1} = \M1^{-1/2} \K1 \M1^{-1/2}$ and $\widetilde{\K2} = \M2^{-1/2} \K2 \M2^{-1/2}$.

Given a vector $v \in \mathbb{R}^{N}$, its $M-$norm is defined as 
$ \left\|v\right\|_M := \sqrt{v^T M v}$.
Note that $\left\|v\right\|_M  = \left\|M^{1/2} v \right\|$. 
With this definition, from \eqref{eq:ADIerr} we infer that
$$ \frac{\left\|e_J\right\|_M}{\left\|e_0\right\|_M} \leq \left\|T_J\right\|. $$
If $\Lambda\left(\M1^{-1} \K1\right) \subseteq [a,b]$ and $\Lambda\left(\M2^{-1} \K2\right) \subseteq [c,d]$, then it holds
\begin{equation} \label{eq:Tjnorm}
\left\|T_J\right\| 
\leq \max_{\lambda \in [a,b], \, \mu \in [c,d] } \prod_{j=1}^{J} \left|\frac{\lambda - \gamma_j}{\lambda + \omega_j} \cdot \frac{\mu - \omega_j}{\mu + \gamma_j} \right|.
\end{equation}

An explicit expression for the parameters $\gamma_j,\omega_j$, $j = 1, \ldots, J$ which minimize the right-hand side of \eqref{eq:Tjnorm} is known \cite{Wachspress1963}.
When these parameters are selected 
the convergence behavior of ADI is well-understood. In particular, if we assume for simplicity that $[a,b] = [c,d]$, then the number of ADI iterations needed to ensure that $\left\|T_J\right\| \leq \epsilon $, for small enough $\epsilon$, is 
\begin{equation} \label{eq:ADIits} J = \left \lceil\frac{1}{\pi^2}
    \ln \left(4 \frac{b}{a}\right) \ln \left(\frac{4}{\epsilon}\right)
    \right \rceil, \end{equation}
where $\left \lceil\cdot  \right \rceil $ denotes the integer round toward positive infinity. We emphasize that the dependence of $J$ on the spectral condition number of the matrices involved is logarithmic, and hence extremely mild. 
Moreover, once a tolerance $\epsilon$ is chosen, the number of ADI iterations can be selected a priori according to \eqref{eq:ADIits}. 

\subsubsection{The algorithm}

A simple trick to reduce the computational cost of each ADI iteration is to define
$$  \tilde{s}_{i} = \left\{\begin{array}{ll} \left( \M1 \otimes I_n \right) s_{i}  & \mbox{for } i \in \mathbb{N}, \\
 \left(I_n \otimes  \M2\right)  s_{i} & \mbox{for } i \notin \mathbb{N}.
\end{array}\right. $$
Equations \eqref{eq:ADI} now read
%
\begin{eqnarray*} 
\left(\left(\K1 + \omega_j \M1 \right) \otimes I_n \right) \tilde{s}_{j - 1/2}  & = & r - \left(I_n \otimes \left(\K2 - \omega_j \M2 \right)\right) \tilde{s}_{j-1},  \\
\left(I_n \otimes \left(\K2 + \gamma_j \M2 \right)\right)  \tilde{s}_{j} & = & r - \left(\left(\K1 - \gamma_j \M1 \right) \otimes I_n  \right)\tilde{s}_{j-1/2}.
\end{eqnarray*} 

We now summarize the steps of the ADI method.

\begin{algorithm}
\caption{ADI method}\label{ADI}
\begin{algorithmic}[1]
\State Fix the number of iterations $J$ according to \eqref{eq:ADIits}
\State Compute the parameters $\omega_j$, $\gamma_j$, $j=1,\ldots,J$.
\State Set $\tilde{s}_0 = 0$.
\For{$j=1,\ldots,J$} 
\State Set $r_{j-1} = r - \left(I_n \otimes \left(\K2 - \omega_j \M2 \right)\right) \tilde{s}_{j-1}$.
\State Solve $\left(\left(\K1 + \omega_j \M1 \right) \otimes I_n \right) \tilde{s}_{j - 1/2}  = r_{j-1} $.
\State Set $ r_{j-1/2} = r - \left(\left(\K1 - \gamma_j \M1 \right) \otimes I_n  \right)\tilde{s}_{j-1/2} $.
\State Solve $ \left(I_n \otimes \left(\K2 + \gamma_j \M2 \right)\right)  \tilde{s}_{j} = r_{j-1/2} $.
\EndFor
\State Set $s_J = \left(\M1 \otimes I_n \right)^{-1} \tilde{s}_J $.
\end{algorithmic}
\end{algorithm}

\subsubsection{ADI as a preconditioner}

Our main interest is to apply ADI as a preconditioner. We observe that, if we take as initial guess $s_0=0$ then equality \eqref{eq:ADIerr} can be rewritten as 
$$ s_{J} = 
M^{-1/2}\left(I - T_J\right) M^{1/2} \p^{-1} r =:
\p_J^{-1} r .
$$
Hence performing $J$ iterations of the ADI method to the system $\p s = r$ is equivalent to multiply $r$ by the matrix $\p_J^{-1} = M^{-1/2}\left(I - T_J\right) M^{1/2} \p^{-1}$.
We emphasize that such equivalency is just a theoretical tool, and it is never used to actually compute $s_J$. We also emphasize that, since the number of iterations $J$ is fixed by the chosen tolerance $\epsilon$, the preconditioner $\p_J^{-1}$ does not change between different CG iterations.

In order to apply ADI as a preconditioner for CG, two issues has to be addressed. 
First, the CG method may break down if an arbitrary preconditioner is considered; in order to safely use the ADI method, we need to show that $\p_J$ is symmetric and positive definite.
Second, the choice of the tolerance $\epsilon$ for ADI is crucial and has to be discussed. Indeed, a tolerance that is too strict yields unnecessary work, while a tolerance that is too loose may compromise the convergence of CG.  

The following theorem addresses both issues. In particular, it presents a nice and simple upper bound for the spectral conditioning of $\p_J^{-1} \A $ in terms of $\epsilon$ and of the conditioning of the exactly preconditioned system $\p^{-1} \A$. 
A proof of this theorem, which generalizes the results of \cite[Section 3]{Wachspress1963}, can be found in \cite{Wachspress2013}. We give a proof of this theorem in our notation, to keep the present manuscript as self-contained as possible.

\begin{thm} \label{ADIprec}

The ADI preconditioner $\p_J$ (with optimal parameters) is positive definite. Moreover, if $\left\|T_J\right\| \leq \epsilon$, then it holds
\begin{equation} \label{eq:eps} \kappa\left( \p^{-1}_{J} \A \right) \leq \left(\frac{1 + \epsilon}{1 - \epsilon}\right) \kappa\left( \p^{-1} \A \right). \end{equation} 

\end{thm}

\begin{proof}

We note that $\p_J$ is symmetric and positive definite if and only if the same holds true for 
\begin{equation} \label{eq:1} M^{1/2} \p_J^{-1} M^{1/2} = \left(I - T_J \right) \widetilde{\p}^{-1}, \end{equation}
with $\widetilde{\p} = M^{-1/2} \p M^{-1/2}$. We have already observed that $T_J$ is symmetric. Moreover, $I - T_J$ is positive definite, since $\left\|T_J\right\| < 1$.
The last inequality follows from \eqref{eq:Tjnorm}; indeed if we take any set of parameters that satisfy $0 < \gamma_j,\omega_j\leq \min \left\{a,c \right\}$, $j=1,\ldots,J$, then each factor of the product is strictly smaller than 1. Since the optimal parameters minimize $\left\|T_J\right\|$, also in this case we have $\left\|T_J\right\| < 1$.

We observe that $ \widetilde{\p} = \widetilde{\K1} \otimes I + I \otimes \widetilde{\K2} $ and $I - T_J$ share the same set of eigenvectors, and since they are both symmetric and positive definite, the product in \eqref{eq:1} is again symmetric and positive definite.

We now turn on inequality \eqref{eq:eps}. Matrix $ \p^{-1}_J \A $ is similar to
$$ \widetilde{\p}^{1/2} M^{1/2} \p^{-1}_J \A M^{-1/2} \widetilde{\p}^{-1/2} = (I -T_J) \left( \widetilde{\p}^{-1/2} M^{-1/2} \A M^{-1/2} \widetilde{\p}^{-1/2} \right), $$
where we used the fact that $\widetilde{\p}^{1/2}$ and $I - T_J$ commute, since they share the same set of eigenvectors.

It holds
\begin{eqnarray} \label{eq:leq}
\lambda_{\max}\left(\p^{-1}_J \A \right)  & = & \lambda_{\max} \left( \left(I-T_J\right)\left(\widetilde{\p}^{-1/2} M^{-1/2} \A M^{-1/2} \widetilde{\p}^{-1/2}\right)\right) \nonumber \\
& \leq & (1 + \epsilon) \; \lambda_{\max}\left(\widetilde{\p}^{-1/2} M^{-1/2} \A M^{-1/2} \widetilde{\p}^{-1/2}\right) \nonumber \\
& = & (1 + \epsilon) \; \lambda_{\max}\left( \p^{-1} \A \right).
\end{eqnarray}
With analogous computations one can show that
\begin{equation} \label{eq:geq} \lambda_{\min}\left(\p^{-1}_J \A \right) \geq (1 - \epsilon) \; \lambda_{\min} \left( \p^{-1} \A \right). \end{equation}
Combining \eqref{eq:geq} and \eqref{eq:leq}, inequality \eqref{eq:eps} is proved.
\end{proof}

\subsubsection{Computational cost}

Lines 1 and 2 of Algorithm \ref{ADI}, namely the computation of $J$ and of the optimal parameters, require an estimate of the minimum and maximum eigenvalue of $\M1^{-1} \K1$ and $\M2^{-1} \K2$. Since all the matrices involved are symmetric, positive definite and banded, this task can be achieved inexpensively, e.g. with a few iterations of the power method.

The computational effort of one full ADI iteration consists in $2n$ matrix-vector products (lines 5 and 7 of Algorithm \ref{ADI}) and the solution of $2n$ linear systems (lines 6 and 8), both of which involve banded matrices of order $n$. Since lines 5 and 7 also require a vector update of order $n^2$, the total cost of a single ADI iteration is roughly $\left(16 p + 10\right) n^2$ FLOPs. 

{\footnotesize
\begin{table}
\begin{center}
\begin{tabular}{|r|r|r|r|r|r|r|r|}
\hline
$p=1$ & $p=2$ & $p=3$ & $p=4$ & $p=5$ & $p=6$ & $p=7$ & $p=8$ \\
\hline
 $3 \cdot 10^5$ & $3 \cdot 10^5$ & $4 \cdot 10^5$ & $4 \cdot 10^5$ & $1 \cdot 10^6$ & $2 \cdot 10^6$ & $2 \cdot 10^6$ & $3 \cdot 10^6$ \\
\hline
\end{tabular}  

\caption{Spectral condition number of $\M1^{-1} \K1$, with $h=2^{-5}$.}
\label{tab:cond}
\end{center}
\end{table}
}

The cost of a single iteration has to be multiplied for the number
iterations $J$, given by \eqref{eq:ADIits}, in order to obtain the
total cost of ADI. For IGA mass and stiffness matrices, it holds $
\displaystyle \frac{b}{a} =  \kappa\left(\M1^{-1} \K1\right) \leq c_p
h^{-2} \approx c_p n^2 $, where $c_p$ is a constant that depends only
on $p$. From inverse estimates for polynomials it follows $c_p \leq p^4$, however no accurate estimate of $c_p$ is known yet.
Hence, we seek numerical evidence of the dependence of
$\kappa\left(\M1^{-1} \K1\right)$ w.r.t. $p$. The results,
obtained for $h=2^{-5}$, are reported in Table \ref{tab:cond}. We can
see that the growth of the conditioning is not dramatic, with a growth
which is weaker than the bound above and,  since we are only
interested in its logarithm, we can conclude that the
number of ADI iterations is robust w.r.t. $p$.

In conclusion the total cost of ADI is roughly $$ \left[\frac{1}{\pi^2} \ln \left(4 c_p n^2 \right) \ln \left(\frac{4}{\epsilon}\right) + 1\right] \left(16 p + 10\right) n^2 $$ FLOPs. We observe  that the cost of the preconditioner is (almost) linear w.r.t. to the number of degrees-of-freedom, and very robust w.r.t. $p$. 
Indeed, this cost has a milder dependence on $p$ than the cost of a matrix-vector product with $\A$, which is $2(2p+1)^2 n^2$ FLOPs.

\section{The 3D case} \label{sec:3D}

When $d=3$, equation \eqref{eq:preceq} takes the form 
\begin{equation} \label{eq:kron3d} \left(\K1 \otimes \M2 \otimes \M3 + \M1 \otimes \K2 \otimes \M3 + \M1 \otimes \M2 \otimes \K3 \right) s = r .\end{equation}

We consider generalizations of the approaches detailed in the previous section, namely the FD direct method and the ADI iterative method. 

Other approaches, which 
however
rely on a low-rank approximation of the right-hand side, can be found in \cite{Grasedyck2004}, \cite{Kressner2010} and \cite{Ballani2013}.

\subsection{The 3D fast diagonalization method}

The direct method presented in section \ref{sec:direct2d} admits a straightforward generalization to the 3D case  (see also \cite{Li2010}, where the Bartels-Stewart approach for the nonsymmetric case is extended to 3D problems). We consider the generalized eigendecompositions
\begin{equation} \label{eq:eigen3d} \K1 \U1 = \M1 \U1 \D1, \qquad \K2
  \U2 = \M2 \U2 \D2, \qquad \K3 \U3 = \M3 \U3 \D3 ,\end{equation} 
with $\D1$, $\D2$, $\D3$ diagonal matrices and
$$ \U1^T \M1 \U1 = I, \qquad \U2^T \M2 \U2 = I, \qquad \U3^T \M3 \U3 = I. $$
Then, \eqref{eq:kron3d} can be factorized as 
$$ \left(\U1 \otimes \U2 \otimes \U3\right)^{-1}\left(\D1 \otimes I \otimes I + I \otimes \D2 \otimes I + I \otimes I \otimes \D3 \right)\left(\U1 \otimes \U2 \otimes \U3 \right)^{-T} s = r,$$
which suggests the following algorithm.

\begin{algorithm}
\caption{FD direct method (3D)}\label{3Ddirect}
\begin{algorithmic}[1]
\State Compute the generalized eigendecompositions \eqref{eq:eigen3d}
\State Compute $\tilde{r} = (\U1 \otimes \U2 \otimes \U3) r $
\State Compute $\tilde{s} = \left(\D1 \otimes I \otimes I + I \otimes \D2 \otimes I + I \otimes I \otimes \D3 \right)^{-1} \tilde{r} $
\State Compute $s = (\U1 \otimes \U2 \otimes \U3)^T \tilde{s} $
\end{algorithmic}
\end{algorithm}

\subsubsection{Computational cost}

Lines 1 and 3 require $O(n^3)$ FLOPs. Lines 2 and 4, as can be
seen by nested applications of formula \eqref{eq:kronprod}, are equivalent to
performing a total of $6$ products between dense matrices of size $n \times n$ and $n \times n^2$. Thus, neglecting lower order terms the overall computational cost of Algorithm \ref{3Ddirect} is $12n^4$ FLOPs.

The direct method is even more appealing in the 3D case than it was in
the 2D case, for at least two reasons. First, the computational cost
associated with the preconditioner setup, that is  the eigendecomposition, is negligible.
This means that the main computational effort of the method consists
in a few (dense) matrix-matrix products, which are level 3 BLAS
operations and typically yield high efficiency thanks to a
dedicated implementation on modern computers by optimized usage of the
memory cache hierarchy  \cite[Chapter 1]{Golub2012}. 
Second, in
a preconditioned CG iteration the cost for applying the preconditioner
has to be compared with the cost of the residual computation (a matrix-vector product
with $\A$) which can be quantified in approximately $2(2p+1)^3 n^3$ for
3D problems, resulting in a FLOPs ratio of the preconditioner application to
residual computation of $ (3n)/(4p^3) \approx n/p^3$. For example, if
$N=256^3$ and $p=4$, the preconditioner requires  only $3$ times more
FLOPs  than the residual computation, while for degree
$p=6$ the matrix-vector product is even more costly than the
preconditioner itself. However in  numerical tests we will see 
that, for all cases of practical interest  in 3D, the computational  time used by the
preconditioner application is far lower that the residual
computation itself. 
This is because the computational time depends not only on the FLOPs
count but also on the memory usage and, as mentioned above, dense matrix-matrix multiplications
greatly benefit of modern computer architecture. 
This approach will show  largely
higher performance than the alternative ADI approach we have
considered.

\subsection{Three-variable ADI} 

Despite the clear advantages presented by the direct method discussed in the
previous section, for the sake of comparison we also consider  ADI. Indeed ADI may benefit from a 
lower FLOPs counting than the direct solver, for large $n$ and low $p$, and 
being an iterative solver we can optimize the target precision as
needed by the  preconditioning step.

However, the ADI  extension to the 3D case is not straightforward. 
Here we follow the iterative scheme proposed by Douglas in \cite{Douglas1962} (see also \cite{Wachspress1994} for a different approach) to solve \eqref{eq:kron3d}:
\begin{subequations} \label{eq:ADI3D}
\begin{eqnarray} 
\left(\mathcal{K}_1 + \omega_j M\right) s_{j-2/3} & = & 2r - \left( \mathcal{K}_1 + 2 \mathcal{K}_2 + 2 \mathcal{K}_3 - \omega_j M  \right) s_{j-1}, \\
 \left( \mathcal{K}_2 + \omega_j M \right) s_{j - 1/3} & = &  \mathcal{K}_2 s_{j-1} + \omega_j M s_{j-2/3} ,\\ 
 \left( \mathcal{K}_3 + \omega_j M \right) s_{j} & = &  \mathcal{K}_3 s_{j-1} + \omega_j M s_{j-1/3},
\end{eqnarray} 
\end{subequations}
where $\mathcal{K}_1 = \K1 \otimes \M2 \otimes \M3$, $\mathcal{K}_2 = \M1 \otimes \K2 \otimes \M3$, $\mathcal{K}_3 = \M1 \otimes \M2 \otimes \K3$, $M = \M1 \otimes \M2 \otimes \M3 $, and the $\omega_j$ are real positive parameters.
After $J$ steps, reasoning as in the 2D case, we can derive an expression for the error $e_J$ similar to \eqref{eq:ADIerr}:
$$ e_J = M^{-1/2} T_J M^{1/2} e_0, $$
where $T_J$ is a symmetric positive definite matrix that depends on $\omega_1, \ldots, \omega_J$. Hence, the relative error in the $M-$norm is bounded by the euclidean norm of $T_J$.

We assume for simplicity that $ \Lambda\left( \M1^{-1} \K1 \right) = \Lambda\left( \M2^{-1} \K2 \right) = \Lambda\left( \M3^{-1} \K3 \right) = \left[a, b\right]$. 
Then it can be shown that
$$ \left\|T_J \right\| \leq \max_{\lambda_1, \lambda_2, \lambda_3 \in \left[a, b\right] } \left| \prod_{j=1}^J \left(1 - 2 \omega_j^2 \frac{ \lambda_1 + \lambda_2 + \lambda_3 }{\left(\omega_j + \lambda_1 \right) \left(\omega_j + \lambda_2 \right) \left(\omega_j + \lambda_3 \right) } \right) \right| =: \rho_J(\omega_1,\ldots,\omega_J). $$
%
Clearly we are interested in choosing the parameters $\omega_j$, $j=1, \ldots, J$ so that the right-hand side of the above inequality is minimized. However, unlike in the 2D case, no expression for the solution of such minmax problem is known, and hence we cannot rely on an optimal choice for the parameters. This makes the ADI approach less appealing than in the 2D case. 
The suboptimal choice proposed in \cite{Douglas1962} still guarantees that the number of iterations $J_0$ needed to ensure that $ \left\|T_{J_0} \right\| \leq \epsilon$ satisfies
\begin{equation} \label{eq:ADI3Diter} J_0 \; \approx \; 1.16 \; \ln\left( \frac{b}{a} \right) \ln\left(\epsilon^{-1} \right) = \displaystyle O \left(\ln\left( \frac{b}{a} \right) \ln\left(\epsilon^{-1} \right) \right), \end{equation}
Note that also in the 2D case we have that the number of iterations required to achieve convergence is $O \left(\ln\left( \frac{b}{a} \right) \ln\left(\epsilon^{-1} \right) \right)$ (cf. \eqref{eq:ADIits}). However, in 3D the constant hidden in this asymptotic estimate is significantly greater than in the 2D case. 

In fact, our numerical experience indicates that typically the condition $\left\| T_J \right\| \leq \epsilon$ is satisfied after much fewer iterations than \eqref{eq:ADI3Diter} would suggest. To avoid unnecessary iterations, we introduce a different stopping criterion, based on the evaluation of the right hand side of \eqref{eq:ADI3D}. 
More precisely, we compute the parameters $\omega_1, \ldots, \omega_{J_0}$ according to \cite{Douglas1962} but then perform only the first $J \leq J_0$ iterations, where $J$ is the smallest index such that $\rho_J (\omega_1, \ldots, \omega_J) \leq \epsilon$.


Another possible approach is to discard the suboptimal choice of \cite{Douglas1962}, and select the parameters $\omega_1, \ldots, \omega_J$ following a greedy strategy, in the spirit of the approach proposed in \cite{Penzl2000} for the two-variables ADI method. After iteration $k$, we compute
\begin{equation} \label{eq:greedy1} \left(\lambda_1^*, \lambda_2^*, \lambda_3^*\right) = \operatornamewithlimits{argmax}_{ \lambda_1, \lambda_2, \lambda_3 \in \left[a, b\right] } \left| \prod_{j=1}^{k} \left(1 - 2 \omega_j^2 \frac{ \lambda_1 + \lambda_2 + \lambda_3 }{\left(\omega_j + \lambda_1 \right) \left(\omega_j + \lambda_2 \right) \left(\omega_j + \lambda_3 \right) } \right) \right|, \end{equation}
and then $\omega_{k+1}$ is defined as the nonnegative number that minimizes the error at $\left(\lambda_1^*, \lambda_2^*, \lambda_3^*\right)$, that is
\begin{equation} \label{eq:greedy2} \omega_{k+1} = \operatornamewithlimits{argmin}_{\omega \geq 0} \left| 1 - 2 \omega^2 \frac{ \lambda_1^* + \lambda_2^* + \lambda_3^* }{\left(\omega + \lambda_1^* \right) \left(\omega + \lambda_2^* \right) \left(\omega + \lambda_3^* \right) } \right|. \end{equation}

To reduce the computational cost of one ADI iteration, we multiply the first equation of \eqref{eq:ADI3D} by $\left(I_n \otimes \M2 \otimes \M3 \right)^{-1}$, the second by $\left(\M1 \otimes I_n \otimes \M3 \right)^{-1}$ and the third by $\left(\M1 \otimes \M2 \otimes I_n \right)^{-1}$. After some algebraic manipulation we obtain
\begin{subequations} \label{eq:ADI3D2}
\begin{eqnarray} 
\left(\left(\K1 + \omega_j \M1 \right) \otimes I_n \otimes I_n \right) s_{j}^{*} & = & \tilde{r} - \left(\left( \K1 - \omega_j \M1 \right) \otimes I_n  \otimes I_n \right) s_{j-1} - 2 \left(\M1 \otimes I_n \otimes I_n\right) \nonumber \\
 \label{eq:ADI3D2a} & & \cdot \left( I_n \otimes \M2^{-1} \K2 \otimes I_n + I_n \otimes I_n \otimes \M3^{-1} \K3  \right) s_{j-1}, \\
 \left( I_n \otimes \left(\K2 + \omega_j \M2 \right) \otimes I_n \right) s_{j}^{**} & = & \left(I_n \otimes \M2  \otimes I_n\right) \nonumber \\
\label{eq:ADI3D2b} & & \cdot \left( \left( I_n \otimes \M2^{-1} \K2 \otimes I_n \right) s_{j-1} + \omega_j s_{j}^*   \right), \\ 
\left( I_n \otimes I_n \otimes \left(\K3 + \omega_j \M3 \right) \right) s_{j} & = & \left(I_n \otimes I_n \otimes \M3 \right) \nonumber \\ 
\label{eq:ADI3D2c} & & \cdot \left( \left( I_n \otimes I_n \otimes \M3^{-1} \K3\right) s_{j-1} + \omega_j s_{j}^{**}   \right),
\end{eqnarray} 
\end{subequations}
where $\tilde{r} = 2\left(I_n \otimes \M2 \otimes \M3 \right)^{-1} r $. Note that the vectors 
$$ u_j := \left( I_n \otimes \M2^{-1} \K2 \otimes I_n \right) s_{j-1}, v_j := \left( I_n \otimes I_n \otimes \M3^{-1} \K3\right) s_{j-1}, $$ 
which both appear twice in \eqref{eq:ADI3D2}, need to be computed only once.
We consider one last trick to save some computational cost. Let 
$$b_j := \left( I_n \otimes I_n \otimes \M3^{-1} \K3\right) s_{j-1} + \omega_j s_{j}^{**},$$
then
$$ v_{j+1} = \left(I_n \otimes I_n \otimes \M3 \right)^{-1} \left( I_n \otimes I_n \otimes \left(\K3 + \omega_j \M3 \right) \right) s_{j} - \omega_j s_j = b_j - \omega_j s_j. $$
where the last equality is a consequence of equation \eqref{eq:ADI3D2c}. Hence we can use the known vectors $b_j$ and $s_j$ to inexpensively compute $v_{j+1}$.
We summarize all these considerations in Algorithm \ref{algo:ADI-3D}.

\begin{algorithm}
\caption{ 3D ADI method}\label{algo:ADI-3D}
\begin{algorithmic}[1]
\State Compute all the eigenvalues of $M_1^{-1} K_1$, $M_2^{-1} K_2$ and $M_3^{-1} K_3$.
\State Compute the number of iterations $J$ and parameters $\omega_1, \ldots, \omega_J$ such that $\rho_J (\omega_1,\ldots,\omega_J) \leq \epsilon $.
\State Compute $\tilde{r} = 2\left(I_n \otimes \M2 \otimes \M3 \right)^{-1} r$.
\State Set $s_0, v_0 = 0$.
\For{$j=1,\ldots,J$} 
\State Solve $\left( I_n \otimes \M2 \otimes I_n \right) u_j = \left( I_n \otimes \K2 \otimes I_n \right) s_{j-1} $
\State Compute the right-hand side of \eqref{eq:ADI3D2a}: 
\State $r^*_j = \tilde{r} - \left(\left( \K1 - \omega_j \M1 \right) \otimes I_n  \otimes I_n \right) s_{j-1} - \left(2 \M1 \otimes I_n \otimes I_n\right) \left(u_j + v_j \right)$
\State Solve $ \left( \left(\K1 + \omega_j \M1 \right) \otimes I_n \otimes I_n \right) s_{j}^* = r_j^*$
\State Compute the right-hand side of \eqref{eq:ADI3D2b}: $ r^{**}_j = \left(I_n \otimes \M2  \otimes I_n\right) \left( u_j + \omega_j s_{j}^*   \right) $
\State Solve $ \left( I_n \otimes \left(\K2 + \omega_j \M2 \right) \otimes I_n \right) s_{j}^{**} = r^{**}_j $
\State Set $b_j = v_j + \omega_j s^{**}_j $
\State Compute the right-hand side of \eqref{eq:ADI3D2c}: $ r_j = \left(I_n  \otimes I_n \otimes \M3 \right) b_j $
\State Solve $ \left( I_n \otimes I_n \otimes \left(\K3 + \omega_j \M3 \right) \right) s_{j} = r_j $
\State Set $v_{j+1} = b_j - \omega_j s_j $
\EndFor
\end{algorithmic}
\end{algorithm}

As for the previous methods, we are mainly interested in using ADI as
a preconditioner. It can be shown that Theorem \ref{ADIprec} holds
also in the 3D case. We do not report the details, as the arguments used
in the proof are the same as in the 2D case. This means that the 3D
ADI method can be used as a preconditioner for CG and that relation
\eqref{eq:eps}  guides in choosing the inner tolerance $\epsilon$.

\subsubsection{Computational cost}
 
At each iteration, the main computational effort is represented by the solution of four linear systems (lines 5, 7, 9 and 12) and five matrix products (lines 5, 6, 8, and 11). As always, by exploiting the Kronecker structure of the matrices involved, each of these computations can be performed at a cost of $2(2p+1) n^3$ FLOPs. A careful analysis reveals that the total cost of a single ADI iteration is $ \left(36 p + 29\right) n^3$ FLOPs, where we neglected terms of order lower than $n^3$.

Unlike in the 2D case, the number of iterations $J$ is not known a priori. However, if we consider $J_0$ in \eqref{eq:ADI3Diter} as an upper bound for the number of iterations, we can bound the total computational cost can be bounded roughly by
\begin{equation} \label{eq:ADI3Dflops} 1.16 \; \ln\left( c_p n^2 \right) \ln\left(\epsilon^{-1} \right) \left(36 p + 29\right) n^3 . \end{equation}
where, as in the 2D case, we replaced $ \displaystyle \frac{b}{a}$ with $c_p n^2$. We observe that asymptotically the complexity of ADI is almost O($n^3$), that is almost linear w.r.t. the number of degrees-of-freedom. On the other hand, a closer look at \eqref{eq:ADI3Dflops} reveals that the number of FLOPs required may be actually quite large, even for small or moderate $p$.
We remark that the sequential nature of the ADI iteration makes it less suited for parallelization than the FD method.

\section{Application to multi-patch problems} \label{sec:multi-patch}

To enhance flexibility in geometry representation, typically
multi-patch parametrizations are adopted in isogeometric analysis. This
means that the domain of interest $\Omega$ is the union of patches
$\Omega_i$ such that $\Omega_i = \mathbf{F}_i([0,1]^d)$, and each
$\mathbf{F}_i$ is a spline (or NURBS) parametrization. Typically  $\Omega_i \cap \Omega_j$ is an empty set,
or a vertex, or the full common edge or the full common face (when
$d=3$) of the
patches. Furthermore we assume that the meshes are conforming, 
that is  for  each patch interface  $\Omega_i \cap \Omega_j$ the isogeometric
functions on $\Omega_i $ and the ones on $\Omega_j$
generate the same trace space. See, e.g., \cite{acta-IGA} for more details.

For such a configuration, we can easily combine  the approaches discussed in
the previous sections with an overlapping Schwarz
preconditioner.  For that, we need to  further split  $\Omega$ into
overlapping subdomains. We choose the subdomains as pairs of neighboring patches merged together. Precisely, let $N_s$ denotes the total number of interfaces
between neighboring patches. We define
$$ \Theta_i = \Omega_{i_1} \cup \Omega_{i_2}, \qquad i=1,\ldots,N_s, $$
where $\Omega_{i_1}$ and $\Omega_{i_2}$ are the patches which share the $i-$th interface. 
It holds $\Omega = \displaystyle \bigcup_i \Theta_i $.

Now let $R_i$ be the rectangular restriction matrix on the
degrees-of-freedom associated with the $i-$th subdomain, and let $\A_i
= R_i \A R_i^T$.The (exact) additive Schwarz preconditioner is
\begin{equation} \label{eq:Peas} \mathcal{P}_{EAS} = \sum_{i = 1}^{N_s} R_i^T \A_i^{-1} R_i, \end{equation}
and its inexact variant
\begin{equation} \label{eq:Pias} \mathcal{P}_{IAS} = \sum_{i = 1}^{N_s} R_i^T \widetilde{\A}_i^{-1} R_i, \end{equation}
where each $\widetilde{A}_i^{-1}$ is a suitable approximation of $\A_i^{-1}$.
Each $\A_i$ represents the system matrix of a discretized Poisson problem on $\Theta_i$. A crucial observation is that, under the conforming assumption, $\Theta_i$ can be considered a single-patch domain. Thus, it is possible to construct a preconditioner of the form \eqref{eq:prec} for $\A_i$, which we denote with $\p_i$. 
Then $\p_i$ can be used to construct $\widetilde{\A}_i^{-1}$ (we can have $\widetilde{\A}_i = \p_i$, or $\widetilde{\A}_i^{-1}$ may represent a fixed number of iteration of some iterative method preconditioned by $\p_i$). 
The proposed approach is somewhat unusual in the context of domain decomposition methods. 
Indeed, it is more common to split the domain $\Omega$ into a large number of small subdomains, so that local problems can be efficiently solved by parallel architectures.
%
Here, on the other hand, the subdomains are chosen so that the basis functions of the local problems have a tensor structure that can be exploited by our preconditioner. 
The efficiency of such preconditioners, which is demonstrated numerically in the next section, makes it feasible to work with local problems whose size is comparable with that of the whole system.
Finally, we remark that the large overlap between neighboring
subdomains ensures that the outer iteration converges fast, independently of $p$ and $h$.

Extension of this approach to nonconforming discretizations would
require the use of nonconforming DD preconditioners (e.g., \cite{Kleiss2012})
instead of an overlapping Schwarz preconditioner. 

\section{Numerical experiments} \label{sec:experiments}

We now numerically show the potential of the approaches described in
Sections \ref{sec:2D} and \ref{sec:3D}. All the algorithms are
implemented in {\sc Matlab} Version 8.5.0.197613 (R2015a), with the
toolbox {\sc GeoPDEs} \cite{DeFalco2011}, on a Intel Xeon i7-5820K processor, running at 3.30 GHz, and with 64 GB of RAM. 
Although the Sylvester-based approaches are very suited for parallelization, particularly the FD method, here we benchmark sequential execution and use only one core for the experiments.

We give a few technical details on how the methods were implemented. For the FD method, we used the {\sc Matlab} function {\tt eig} to compute the generalized eigendecomposition \eqref{eq:eigen2d} and \eqref{eq:eigen3d}. For 2D ADI, the number of iterations was set according to \eqref{eq:ADIits}. The extreme eigenvalues of $\M1^{-1} \K1$ and $\M2^{-1} \K2$, which are required for computing the optimal parameters derived in \cite{Wachspress1963}, were approximated using 10 iterations of the (direct and inverse) power method. For 3D ADI, the eigenvalue computation necessary to select the parameters was again performed using {\tt eig}. In both 2D and 3D, at each ADI iteration the linear systems were solved using {\sc Matlab}'s direct solver ``backslash''. 
Finally, in 3D algorithms the products involving Kronecker matrices were performed using the function from the free {\sc Matlab} toolbox Tensorlab \cite{Sorber2014}. 

Although in many of the problems considered here the matrix pencils $(\K1,\M1)$, $(\K2,\M2)$ (and in the 3D case $(\K3,\M3)$) coincide, we never exploit this fact in our tests. For example, in line 1 of Algorithm 1 we always compute two eigendecompositions even if in the current problem we have $\M1 = \M2$ and $\K1 = \K2$. In this way, the computational effort reflects the more general case in which such matrices are different.

\subsection{2D experiments}

We start by considering 2D problems, and observe the performance of the FD and ADI methods on four test problems with different geometries: a square, a quarter of ring, a stretched square and plate with hole. In all problems we set $K$ as the identity matrix, since according to Proposition \ref{prop:bound1} 
the presence of coefficients and of a nontrivial geometry have an analogous impact on the difficulty of the problem.

The square domain is simply $[0,1]^2$, and the other domains are shown in Figure \ref{fig:2Dfigures}. 
In the case of the plate with hole we chose the same parametrization considered in \cite[Section 4.2]{Cottrell2009}. In particular, two control points are placed in the same spacial location, namely the left upper corner, and this creates a singularity in the Jacobian of $\F$.
Thus, in this case the bound provided by Proposition \ref{prop:bound1} becomes $\kappa\left(\p^{-1}
  \A\right)\leq +\infty$, and it is hence useless.
In principle, our approaches could perform arbitrarily bad and this
problem is indeed intended to test their performance in this unfavorable case. 
In all problems, except the last one, the system $\A u = b$ represents the discretization of problem \eqref{poisson}, with $f = 2\left(x^2- x \right) + 2 \left(y^2 - y\right)$. For the plate with hole domain, we considered $f = 0$ and mixed boundary conditions.  

\begin{figure}
\begin{center}
  \begin{minipage}[b]{0.32\textwidth}
    \includegraphics[trim=9cm 0cm 9cm 0cm, clip=true, width=\textwidth]{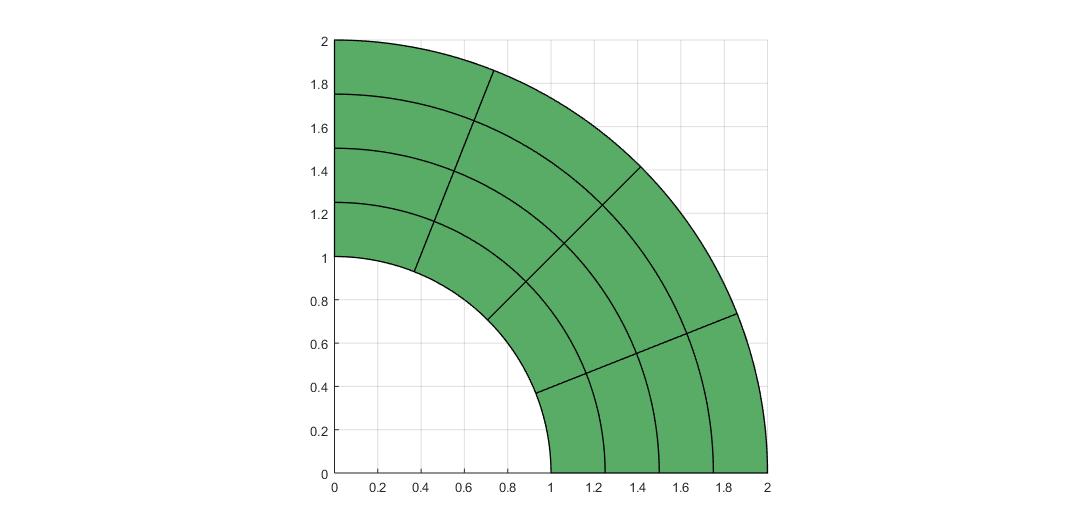}
  \end{minipage}
  \begin{minipage}[b]{0.32\textwidth}
    \includegraphics[trim=9cm 0cm 9cm 0cm, clip=true, width=\textwidth]{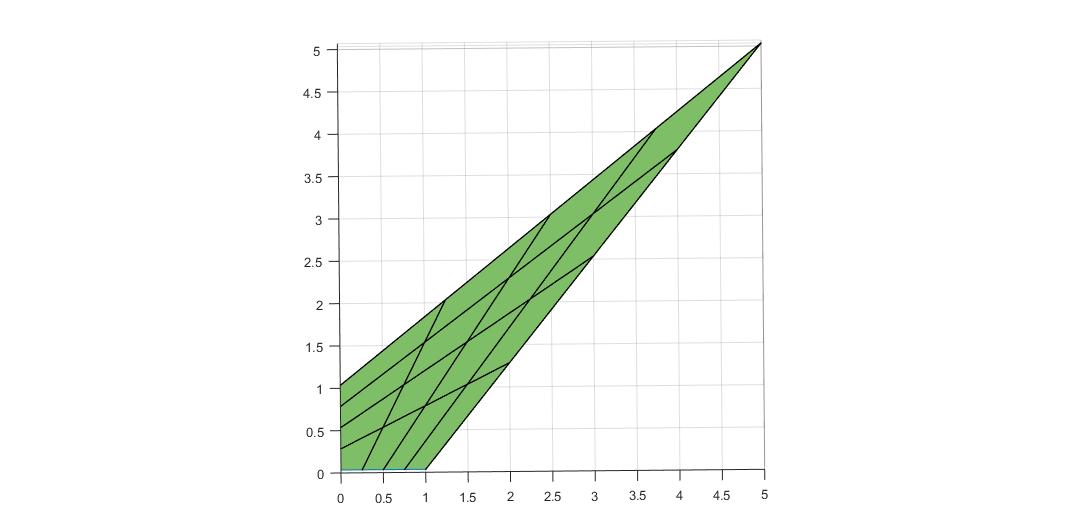}
  \end{minipage}
	  \begin{minipage}[b]{0.32\textwidth}
    \includegraphics[trim=9cm 0cm 9cm 0cm, clip=true, width=\textwidth]{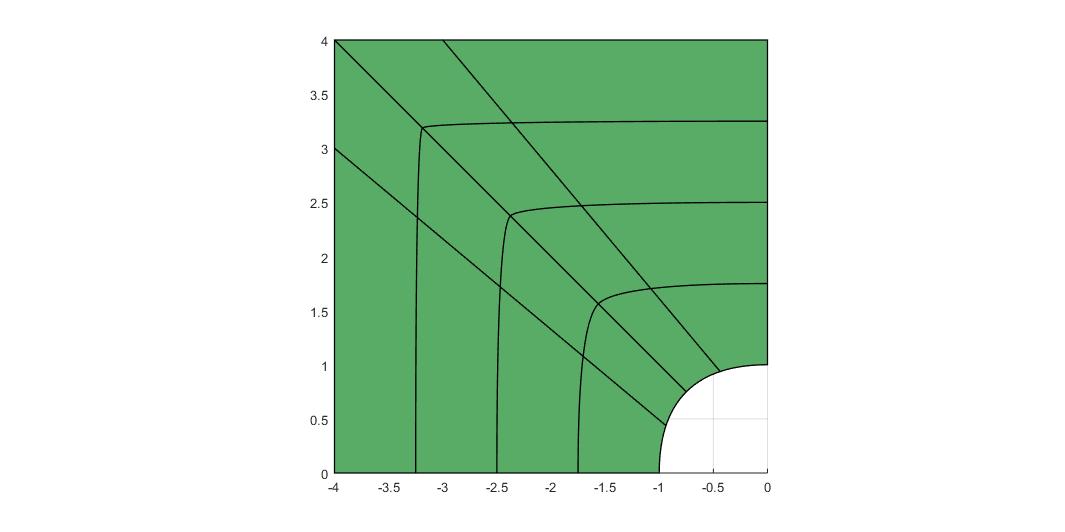}
  \end{minipage}
\end{center}
\caption{Quarter of ring, stretched square and plate with hole domains}
\label{fig:2Dfigures}
\end{figure}

We start by considering the problem on the square. As already said, in
this case $\p = \A$ and we can directly apply the considered method to
the system $\A u = b$. This is not a realistic case but serves as a
preliminary check on the proposed theory and implementations. 
Results are shown in Table \ref{tab:square}; in the upper part we report the CPU time for the FD method, while in the lower part we report the number of iterations and CPU time for ADI, whose tolerance was set to $10^{-8}$, for different values of $h$ and $p$. 

We observe that the computation time for the direct method is
substantially independent w.r.t. $p$; fluctuations in time appearing
in the finer discretization levels are due to the {\tt eig} function,
which constitute the main computational effort of the method in our implementation. 
Similarly, computation times for ADI do not change significantly by varying $p$, and fluctuations are due to {\sc Matlab}'s direct solver.

Concerning the dependence on $h$, based on the analysis of computational cost we expect ADI to perform better than FD for small enough $h$. That is indeed what we can see in the experimental results; however, ADI starts outperforming FD for a very small value of $h$, corresponding roughly to 16 million degrees-of-freedom. 
Moreover, we emphasize that the CPU times of the two methods are comparable for all the discretization levels considered.

\begin{table}
\begin{center}
\footnotesize
\begin{tabular}{|r|c|c|c|c|c|c|}
\hline
& \multicolumn{6}{|c|}{ FD Direct Solver Time (sec)} \\
\hline
$h^{-1}$ & $p=1$ & $p=2$ & $p=3$ & $p=4$ & $p=5$ & $p=6$ \\
\hline
512 & \z\z0.18 & \z\z0.16 & \z\z\z0.16 & \z\z\z0.16 & \z\z0.16 & \z\z0.15 \\
\hline
1024 & \z\z1.52 & \z\z1.17 & \z\z\z1.02 & \z\z\z0.94 & \z\z0.95 & \z\z0.89 \\
\hline
2048 & \z10.62 & \z10.04 & \z\z11.70 & \z\z\z9.21 & \z\z7.64 & \z\z6.68 \\
\hline
4096 & \z72.73 & \z71.72 & \z127.42 & \z108.91 & \z68.91 & \z83.83 \\
\hline
8192 & 511.33 & 511.27 & 1145.04 & 1030.96 & 515.40 & 856.62 \\
\hline
\end{tabular} 
 
\vspace{0.5cm}

\begin{tabular}{|r|c|c|c|c|c|c|}
\hline
& \multicolumn{6}{|c|}{ ADI Iterations / Time (sec)} \\
\hline
$h^{-1}$ & $p=1$ & $p=2$ & $p=3$ & $p=4$ & $p=5$ & $p=6$ \\
\hline
512 & 29 / \z\z0.34 & 28 / \z\z0.33 & 29 / \z\z0.37 & 30 / \z\z0.40 & 31 / \z\z0.43 & 32 / \z\z0.45 \\
\hline
1024 & 31 / \z\z1.72 & 31 / \z\z1.56 & 32 / \z\z1.64 & 33 / \z\z1.82 & 34 / \z\z1.96 & 35 / \z\z2.05 \\
\hline
2048 & 34 / \z\z8.39 & 34 / \z11.61 & 35 / \z\z8.39 & 36 / \z10.42 & 37 / \z10.62 & 37 / \z\z9.23 \\
\hline
4096 & 37 / \z37.25 & 37 / \z52.59 & 37 / \z37.48 & 38 / \z40.42 & 39 / \z43.03 & 40 / \z40.25 \\
\hline
8192 & 40 / 160.91 & 39 / 218.11 & 40 / 161.57 & 41 / 173.67 & 42 / 186.61 & 43 / 172.50 \\
\hline
\end{tabular} 
\caption{Square domain. Performance of FD (upper table) and of ADI (lower table)}
\label{tab:square}
\end{center}
\end{table}




The number of iterations of ADI is determined a priori, according to
\eqref{eq:ADIits}, and no a posteriori stopping criterion is
considered.  Table \ref{tab:relerr} 
reports  the relative error $\left\|e_J\right\|_M /
\left\|e_0\right\|_M$ for all cases considered in Table
\ref{tab:square}: observe that in all cases, this value is below the
prescribed tolerance $10^{-8}$ and at the same time,  never smaller than $3 \cdot 10^{-9}$, showing that \eqref{eq:ADIits} is indeed a good choice. 

\begin{table}
\footnotesize
\begin{center}
\begin{tabular}{|r|c|c|c|c|c|c|}
\hline
$h^{-1}$ & $p=1$ & $p=2$ & $p=3$ & $p=4$ & $p=5$ & $p=6$ \\
\hline
512 & $3.0 \cdot 10^{-9}$ & $7.0 \cdot 10^{-9}$ & $6.8 \cdot 10^{-9}$ & $6.4 \cdot 10^{-9}$ & $6.7 \cdot 10^{-9}$ & $6.7 \cdot 10^{-9}$ \\
\hline
1024 & $7.7 \cdot 10^{-9}$ & $7.2 \cdot 10^{-9}$ & $5.4 \cdot 10^{-9}$ & $6.2 \cdot 10^{-9}$ & $6.3 \cdot 10^{-9}$ & $5.8 \cdot 10^{-9}$ \\
\hline
2048 & $7.5 \cdot 10^{-9}$ & $6.0 \cdot 10^{-9}$ & $4.9 \cdot 10^{-9}$ & $6.0 \cdot 10^{-9}$ & $5.1 \cdot 10^{-9}$ & $9.7 \cdot 10^{-9}$ \\
\hline
4096 & $6.9 \cdot 10^{-9}$ & $4.9 \cdot 10^{-9}$ & $7.6 \cdot 10^{-9}$ & $9.2 \cdot 10^{-9}$ & $7.8 \cdot 10^{-9}$ & $8.4 \cdot 10^{-9}$ \\
\hline
8192 & $6.1 \cdot 10^{-9}$ & $8.2 \cdot 10^{-9}$ & $7.5 \cdot 10^{-9}$ & $9.3 \cdot 10^{-9}$ & $7.9 \cdot 10^{-9}$ & $8.2 \cdot 10^{-9}$ \\
\hline
\end{tabular} 
\caption{ADI relative error $\left\|e_J\right\|_M / \left\|e_0\right\|_M$ at the final iteration for problems in Table \ref{tab:square}}
\label{tab:relerr}
\end{center}
\end{table}

We now turn to the first two problems with nontrivial geometry, namely the quarter of ring and the stretched square, and employ FD and ADI as preconditioners for CG (represented respectively by matrices $\p$ and $\p_J$). For both problems, we set $\epsilon = 10^{-1}$ as tolerance for the ADI preconditioner. 
%
%
We remark that a slightly better performance could be obtained by an adaptive choice of $\epsilon$, as described in \cite[Chapter 3]{Wachspress2013}. However, we did not implement this strategy.

To better judge the efficiency of the Sylvester-based approaches, we compare the results with those obtained when using a preconditioner based on the Incomplete Cholesky (IC) factorization (implemented by the {\sc Matlab} function {\tt ichol}). 
To improve the performance of this approach, we considered some preliminary reorderings of $\A$, namely those implemented by the {\sc Matlab} functions {\tt symrcm}, {\tt symamd} and {\tt colperm}. The reported results refer to the {\tt symrcm} reordering, which yields the best performance. 
We remark that incomplete factorizations have
been considered as preconditioners for IGA problems in
\cite{Collier2013}, where the authors observed that this approach is
quite robust w.r.t. $p$.

In Tables \ref{tab:quarterofring} and \ref{tab:stretchedsquare}, we report the total computation time and the number of CG iterations for both problems; when ADI is used, we also report the number of iterations performed at each application of the preconditioner. Here and throughout, the computation time includes the time needed to setup the preconditioner. Note that the considered values of the mesh size are larger than in the square case. Indeed, while in the square case we need to store only the blocks $\M1$, $\M2$, $\K1$ and $\K2$, in the presence of nontrivial geometry the whole matrix $\A$ has to be stored, and this is unfeasible for our computer resources when $h$ is too small. Below we report some comments on the numerical results for the two
geometries: the quarter of ring and the stretched square.

\begin{table}
\begin{center}
\footnotesize
\begin{tabular}{|r|c|c|c|c|c|}
\hline
& \multicolumn{4}{|c|}{ CG + $\mathcal{P}$ Iterations / Time (sec)} \\
\hline
$h^{-1}$ & $p=2$ & $p=3$ & $p=4$ & $p=5$  \\
\hline
128 & 25 / 0.04 & 25 / 0.06 & 25 / 0.07 & 25 / \z0.09 \\
\hline
256 & 25 / 0.20 & 25 / 0.26 & 25 / 0.34 & 25 / \z0.40 \\
\hline
512 & 26 / 1.13 & 26 / 1.36 & 26 / 1.62 & 26 / \z2.00 \\
\hline
1024 & 26 / 7.30 & 26 / 8.13 & 26 / 9.09 & 26 / 10.52 \\
\hline
\end{tabular}  

\vspace{0.5cm}

\begin{tabular}{|r|c|c|c|c|c|}
\hline
& \multicolumn{4}{|c|}{ CG + $\mathcal{P}_J$ Iterations / Time (sec)} \\
\hline
$h^{-1}$ & $p=2$ & $p=3$ & $p=4$ & $p=5$  \\
\hline
128 & 25 (5) / 0.12 & 25 (5) / \z0.14 & 25 (5) / \z0.17 & 25 (5) / \z0.19 \\
\hline
256 & 26 (5) / 0.49 & 26 (5) / \z0.54 & 26 (6) / \z0.76 & 25 (6) / \z0.82 \\
\hline
512 & 27 (6) / 2.39 & 27 (6) / \z2.68 & 26 (6) / \z2.94 & 26 (6) / \z3.38 \\
\hline
1024 & 27 (6) / 9.94 & 27 (6) / 11.00 & 27 (7) / 14.15 & 27 (7) / 16.15 \\
\hline
\end{tabular} 

\vspace{0.5cm}

\begin{tabular}{|r|c|c|c|c|}
\hline
& \multicolumn{4}{|c|}{ CG + IC Iterations / Time (sec)} \\
\hline
$h^{-1}$ & $p=2$ & $p=3$ & $p=4$ & $p=5$  \\
\hline
128 & 65 / \z0.12 & 49 / \z0.18 & 40 / \z0.21 & 33 / \z\z0.28 \\
\hline
256 & 130 / \z0.92 & 98 / \z1.16 & 80 / \z1.51 & 65 / \z\z1.87 \\
\hline
512 & 264 / \z7.94 & 198 / \z9.47 & 160 / 11.42 & 128 / \z13.29 \\
\hline
1024 & 533 / 64.54 & 399 / 75.22 & 324 / 89.29 & 262 / 103.26 \\
\hline
\end{tabular} 
\caption{Quarter of ring domain. Performance of CG preconditioned by FD (upper table), by ADI (middle table) and by Incomplete Cholesky (lower table).} 
\label{tab:quarterofring}
\end{center}
\end{table}

\begin{table}
\begin{center}
\footnotesize
\begin{tabular}{|r|c|c|c|c|c|}
\hline
& \multicolumn{5}{|c|}{ CG + $\mathcal{P}$ Iterations / Time (sec)} \\
\hline
$h^{-1}$ & $p=1$ & $p=2$ & $p=3$ & $p=4$ & $p=5$  \\
\hline
128 & 58 / \z0.08 & 61 / \z0.08 & 62 / \z0.12 & 61 / \z0.16 & 61 / \z0.21 \\
\hline
256 & 64 / \z0.39 & 66 / \z0.48 & 66 / \z0.61 & 66 / \z0.82 & 66 / \z0.98 \\
\hline
512 & 69 / \z2.58 & 70 / \z2.83 & 69 / \z3.33 & 69 / \z4.01 & 69 / \z4.88 \\
\hline
1024 & 73 / 18.45 & 73 / 18.56 & 72 / 20.88 & 72 / 23.46 & 71 / 26.94 \\
\hline
\end{tabular}  

\vspace{0.5cm}

\begin{tabular}{|r|c|c|c|c|c|}
\hline
& \multicolumn{5}{|c|}{ CG + $\mathcal{P}_J$ Iterations / Time (sec)} \\
\hline
$h^{-1}$ & $p=1$ & $p=2$ & $p=3$ & $p=4$ & $p=5$  \\
\hline
128 & 58 (5) / \z0.25 & 61 (5) / \z0.28 & 61 (5) / \z0.33 & 61 (5) / \z0.40 & 61 (5) / \z0.48 \\
\hline
256 & 65 (5) / \z0.94 & 65 (5) / \z1.23 & 65 (5) / \z1.35 & 65 (6) / \z1.87 & 66 (6) / \z2.14 \\
\hline
512 & 69 (6) / \z5.76 & 70 (6) / \z6.27 & 69 (6) / \z6.85 & 69 (6) / \z7.71 & 69 (6) / \z8.85 \\
\hline
1024 & 74 (6) / 28.53 & 73 (6) / 27.17 & 73 (6) / 29.25 & 72 (7) / 37.47 & 72 (7) / 42.71 \\
\hline
\end{tabular}  

\vspace{0.5cm}

\begin{tabular}{|r|c|c|c|c|c|}
\hline
& \multicolumn{5}{|c|}{ CG + IC Iterations / Time (sec)} \\
\hline
$h^{-1}$ & $p=1$ & $p=2$ & $p=3$ & $p=4$ & $p=5$  \\
\hline
128 & 50 / \z0.07 & 38 / \z0.07 &  26 / \z0.11 & 22 / \z0.14 & 20 / \z0.21 \\
\hline
256 & 102 / \z0.46 & 88 / \z0.64 & 53 / \z0.68 & 44 / \z0.94 & 38 / \z1.24 \\
\hline
512 & 213 / \z3.91 & 232 / \z7.17 & 115 / \z5.73 & 89 / \z6.99 & 76 / \z8.45 \\
\hline
1024 & 439 / 33.68 & 780 / 93.07 & 248 / 47.56 & 181 / 50.80 & 153 / 63.52 \\
\hline
\end{tabular}  
\caption{Stretched square domain. Performance of CG preconditioned by FD (upper table), by ADI (middle table) and by Incomplete Cholesky (lower table).}
\label{tab:stretchedsquare}
\end{center}
\end{table}

\begin{itemize}
\item For both the ADI and the FD preconditioners, the number of CG iterations is practically independent on $p$ and slightly increases as the mesh is refined, but stays uniformly bounded according to Proposition \ref{prop:bound1}. Moreover, in both approaches the computation times depend on $p$ only mildly.
\item The inexact application of $\p$ via the ADI method does not significantly affect the number of CG iterations. Moreover, the number of inner ADI iterations is roughly the same in all considered cases.
\item The overall performance obtained with FD is slightly better than with ADI for all the considered discretization levels; if finer meshes are considered, ADI should eventually outperform FD.
\item Interestingly, in the IC approach the number of CG iterations decreases for higher $p$. On the other hand, the CPU time still increases due to the greater computational cost of forming and applying the preconditioner. 
\item For small enough $h$, both the ADI and the FD preconditioners yield better performance, in terms of CPU time, than the IC preconditioner.
\end{itemize}

Finally, in Table \ref{tab:platewithhole} we report the results for the plate with hole domain. 
As expected, the performance of the Sylvester-based preconditioners in this case is much worse than in the
previous cases, and in particular they are not robust neither
w.r.t. $h$ nor w.r.t. $p$.  
One can introduce modifications of $\p$  that
significantly improve the conditioning of the preconditioned system,
however we postpone this investigation to a further work.
Interestingly, however, if we compare the results with the ones obtained with the IC preconditioner, we see that computation times are still 
comparable with those relative to the FD preconditioner for all discretization levels.
In conclusion, even in most penalizing case among those considered, the proposed preconditioner is competitive with a standard one.

\begin{table}
\begin{center}
\footnotesize
\begin{tabular}{|r|c|c|c|c|}
\hline
& \multicolumn{4}{|c|}{ CG + $\mathcal{P}$ Iterations / Time (sec)} \\
\hline
$h^{-1}$ & $p=2$ & $p=3$ & $p=4$ & $p=5$  \\
\hline
128 & 125 / \z0.18 & 155 / \z\z0.27 & 186 / \z\z0.48 & 216 / \z\z0.74 \\
\hline
256 & 189 / \z1.33 & 236 / \z\z2.10 & 280 / \z\z3.30 & 320 / \z\z4.56 \\
\hline
512 & 279 / 10.59 & 345 / \z15.83 & 406 / \z21.71 & 446 / \z29.45 \\
\hline
1024 & 404 / 99.15 & 487 / 140.64 & 556 / 174.55 & 587 / 203.66 \\
\hline
\end{tabular}  

\vspace{0.5cm}

\begin{tabular}{|r|c|c|c|c|}
\hline
& \multicolumn{4}{|c|}{ CG + $\mathcal{P}_J$ Iterations / Time (sec)} \\
\hline
$h^{-1}$ & $p=2$ & $p=3$ & $p=4$ & $p=5$  \\
\hline
128 & 200 (5) / \z\z0.95 & 260 (5) / \z\z1.40 & 324 (5) / \z\z2.16 & 380 (6) / \z\z3.49 \\
\hline
256 & 346 (6) / \z\z7.34 & 429 (6) / \z10.25 & 522 (6) / \z14.91 & 620 (6) / \z20.66 \\
\hline
512 & 582 (6) / \z49.93 & 714 (6) / \z68.24 & 846 (6) / \z92.16 & 870 (7) / 127.39   \\
\hline
1024 & 771 (7) / 319.56 & 962 (7) / 438.11 & 1184 (7) / 605.72 & 1366 (7) / 834.23  \\
\hline
\end{tabular}  

\vspace{0.5cm}

\begin{tabular}{|r|c|c|c|c|}
\hline
& \multicolumn{4}{|c|}{ CG + IC Iterations / Time (sec)} \\
\hline
$h^{-1}$ & $p=2$ & $p=3$ & $p=4$ & $p=5$  \\
\hline
128 & \z92 / \z0.15 & \z55 / \z\z0.17 & \z45 / \z\z0.24 & \z38 / \z\z0.33 \\
\hline
256 & 180 / \z1.27 & 114 / \z\z1.37 & \z90 / \z\z1.70 & \z73 / \z\z2.08 \\
\hline
512 & 354 / 10.62 & 253 / \z12.13 & 193 / \z13.64 & 151 / \z15.40 \\
\hline
1024 & 695 / 86.53 & 597 / 111.67 & 444 / 121.18 & 339 / 132.59 \\
\hline
\end{tabular}  
\caption{Plate with hole domain. Performance of CG preconditioned by FD (upper table), by ADI (middle table) and by Incomplete Cholesky (lower table).}
\label{tab:platewithhole}
\end{center}
\end{table}

\subsection{3D experiments}

As in the 2D case, we first consider a domain with trivial geometry, namely the unit cube $[0, 1]^3$, and then turn to more complicated domains, which are shown in Figure \ref{fig:3Dfigures}. 
The first one is a thick quarter of ring; note that this solid has a trivial geometry on the third direction.
The second one is the solid of revolution obtained by the 2D quarter of ring. Specifically, we performed a $\pi/2$ revolution around the axis having direction $(0, 1, 0)$ and passing through $(-1, -1, -1)$. We emphasize that here the geometry is nontrivial along all directions. 
In the cube case, we set $b = ${\tt randn}$(n^3,1)$ for computational ease, while in the other two cases $b$ is the vector representing the function $f(x,y,z) = 2\left(x^2 - x\right) + 2\left(y^2 - y\right) + 2\left(z^2 - z\right) $.
We again set $K$ as the identity matrix in all cases.

\begin{figure}
\begin{center}
  \begin{minipage}[b]{0.49\textwidth}
    \includegraphics[trim=0cm 0cm 0cm 0cm, clip=true, width=\textwidth]{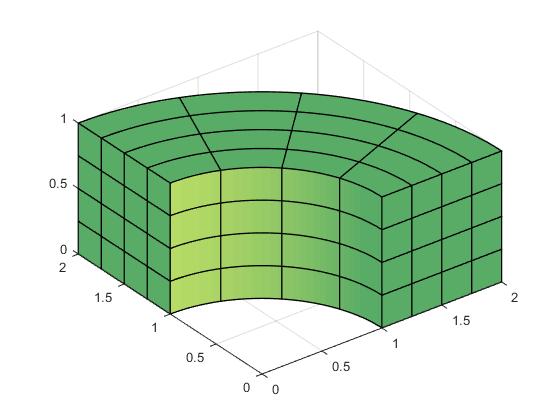}
  \end{minipage}
  \begin{minipage}[b]{0.49\textwidth}
    \includegraphics[trim=0cm 0cm 0cm 0cm, clip=true, width=\textwidth]{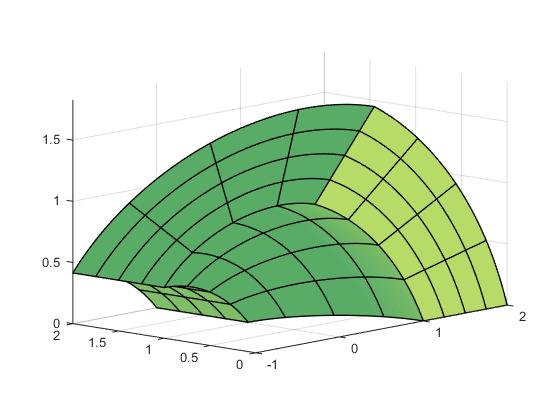}
  \end{minipage}
\caption{Thick ring and revolved ring domains}
\label{fig:3Dfigures}
\end{center}
\end{figure}

We report in Table \ref{tab:cube} the performances of FD and ADI (where the parameters are chosen as in \cite{Douglas1962}) relative to the cube domain. 
We can see that the computational time required by FD is independent of the degree $p$. 
In fact, the timings look impressive and show the great efficiency of this approach. 
We emphasize that, on the finest discretization level, problems with more than one billion variables are solved in slightly more than five minutes, regardless of $p$. 
On the other hand, the ADI solver shows a considerably worse performance. Indeed, while the results confirm that this approach is robust w.r.t. $h$ and $p$, the timings are always a couple of orders of magnitude greater than those obtained with FD. 
A comparison with Table \ref{tab:square} shows also that, as expected, the number of ADI iterations is higher than in the 2D case. 

We mention that, for ADI, we also tested with the greedy choice of the parameters defined by \eqref{eq:greedy1}-\eqref{eq:greedy2}.
This choice yields about 10\%-20\% less iterations than the standard approach. Despite being an effective strategy, the improvement is not dramatic and the FD method is still much more efficient. In the following tests, we always consider the parameters from \cite{Douglas1962}.

\begin{table}
\begin{center}
\footnotesize
\begin{tabular}{|r|c|c|c|c|c|c|}
\hline
& \multicolumn{6}{|c|}{ FD Direct Solver Time (sec) } \\
\hline
$h^{-1}$ & $p=1$ & $p=2$ & $p=3$ & $p=4$ & $p=5$ & $p=6$ \\
\hline
128  & \z\z0.19 & \z\z0.15 & \z\z0.17 & \z\z0.17 & \z\z0.18 & \z\z0.17 \\
\hline
256  & \z\z1.80 & \z\z1.87 & \z\z2.05 & \z\z1.90 & \z\z2.10 & \z\z1.82 \\
\hline
512 & \z23.02 & \z22.58 & \z23.45 & \z23.26 & \z23.65 & \z21.89 \\
\hline
1024 & 331.01 & 316.15 & 328.65 & 318.42 & 331.06 & 310.31 \\
\hline
\end{tabular}  

\vspace{0.5cm}

\begin{tabular}{|r|c|c|c|c|c|c|}
\hline
& \multicolumn{6}{|c|}{ 3D ADI Iterations / Time (sec) } \\
\hline
$h^{-1}$ & $p=1$ & $p=2$ & $p=3$ & $p=4$ & $p=5$ & $p=6$ \\
\hline
128 & 57/\z\z16.91 & 57/\z\z22.50 & 57/\z\z18.31 & 65/\z\z22.12 & 65/\z\z24.30 & 65/\z\z22.35 \\
\hline
256 & 65/\z177.87 & 65/\z256.44 & 65/\z180.94 & 65/\z189.45 & 73/\z239.31 & 73/\z199.29 \\
\hline
512 & 73/1872.01 & 73/2350.32 & 73/2239.23 & 73/2354.14 & 81/2714.25 & 81/1708.36 \\
\hline
\end{tabular}  
\caption{Cube domain. Performance of FD (upper table) and of ADI (lower table). 
We did not run ADI on the finest level due to memory limitations. 
}
\label{tab:cube}
\end{center}
\end{table}




We now consider the problems with nontrivial geometries, where the two methods are used as preconditioners for CG. In the case of ADI, we again set $\epsilon = 10^{-1}$ for both problems. As in the 2D case, we also consider a standard Incomplete Cholesky (IC) preconditioner (no reordering is used in this case, as the resulting performance is better than when using the standard reorderings available in {\sc Matlab}).

In Table \ref{tab:thick_ring} we report the results for the thick quarter ring while in Table \ref{tab:rev_ring} we report the results for the revolved ring. The symbol ``*'' denotes the cases in which even assembling the system matrix $\A$ was unfeasible due to memory limitations. 
From these results, we infer that most of the conclusions drawn for the 2D case still hold in 3D.  
In particular, both Sylvester-based preconditioners yield a better performance than the IC preconditioner, especially for small $h$. 

\begin{table}
\begin{center}
\footnotesize
\begin{tabular}{|r|c|c|c|c|c|}
\hline
& \multicolumn{5}{|c|}{ CG + $\mathcal{P}$ Iterations / Time (sec)} \\
\hline
$h^{-1}$ & $p=2$ & $p=3$ & $p=4$ & $p=5$ & $p=6$ \\
\hline
32 & 26 / \z0.19 & 26 / \z0.38 & 26 / \z0.75 & 26 / \z1.51 & 26 / \z2.64 \\
\hline
64 & 27 / \z1.43 & 27 / \z3.35 & 27 / \z6.59 & 27 / 12.75 & 27 / 21.83 \\
\hline
128 & 28 / 14.14 & 28 / 32.01 & 28 / 61.22 & * & * \\
\hline
\end{tabular}  

\vspace{0.5cm}

\begin{tabular}{|r|c|c|c|c|c|}
\hline
& \multicolumn{5}{|c|}{ CG + $\mathcal{P}_J$ Iterations / Time (sec)} \\
\hline
$h^{-1}$ & $p=2$ & $p=3$ & $p=4$ & $p=5$ & $p=6$ \\
\hline
32 & 26 (7) / \z0.88 & 26 (7) / \z1.20 & 26 (7) / \z\z1.71 & 26 (7) / \z2.62 & 27 (8) / \z4.08 \\
\hline
64 & 27 (7) / \z7.20 & 27 (8) / 10.98 & 27 (8) / \z14.89 & 27 (8) / 21.81 & 27 (8) / 30.56 \\
\hline
128 & 28 (8) / 99.01 & 28 (8) / 98.39 & 28 (8) / 143.45 & * & * \\
\hline
\end{tabular}  

\vspace{0.5cm}

\begin{tabular}{|r|c|c|c|c|c|}
\hline
& \multicolumn{5}{|c|}{ CG + IC Iterations / Time (sec)} \\
\hline
$h^{-1}$ & $p=2$ & $p=3$ & $p=4$ & $p=5$ & $p=6$ \\
\hline
32 & 21 / \z\z0.37 & 15 / \z\z1.17 & 12 / \z\z3.41 & 10 / \z9.43 & \z9 / \z24.05 \\
\hline
64 & 37 / \z4.26 & 28 / \z13.23 & 22 / \z33.96 & 18 / 88.94 & 16 / 215.31 \\
\hline
128 & 73 / 65.03 & 51 / 163.48 & 41 / 385.54 & * & * \\
\hline
\end{tabular}  
\caption{Thick quarter of ring domain. Performance of CG preconditioned by the direct method (upper table), by ADI (middle table) and by Incomplete Cholesky (lower table).}
\label{tab:thick_ring}
\end{center}
\end{table}

\begin{table}
\begin{center}
\footnotesize
\begin{tabular}{|r|c|c|c|c|c|}
\hline
& \multicolumn{5}{|c|}{ CG + $\mathcal{P}$ Iterations / Time (sec)} \\
\hline
$h^{-1}$ & $p=2$ & $p=3$ & $p=4$ & $p=5$ & $p=6$ \\
\hline
32 & 40 / \z0.27 & 41 / \z0.63 & 41 / \z\z1.24 & 42 / \z2.38 & 42 / \z4.13 \\
\hline
64 & 44 / \z2.30 & 44 / \z5.09 & 45 / \z10.75 & 45 / 20.69 & 45 / 35.11 \\
\hline
128 & 47 / 23.26 & 47 / 55.34 & 47 / 101.94 & * & * \\
\hline
\end{tabular}  

\vspace{0.5cm}

\begin{tabular}{|r|c|c|c|c|c|}
\hline
& \multicolumn{5}{|c|}{ CG + $\mathcal{P}_J$ Iterations / Time (sec)} \\
\hline
$h^{-1}$ & $p=2$ & $p=3$ & $p=4$ & $p=5$ & $p=6$ \\
\hline
32 & 40 (7) / \z\z1.39 & 41 (7) / \z\z1.93 & 41 (7) / \z\z2.67 & 42 (7) / \z4.17 & 42 (8) / \z6.25 \\
\hline
64 & 44 (7) / \z11.82 & 44 (8) / \z16.96 & 45 (8) / \z24.31 & 45 (8) / 35.76 & 45 (8) / 49.89 \\
\hline
128 & 47 (8) / 170.69 & 47 (8) / 168.45 & 47 (9) / 239.07 & * & * \\
\hline
\end{tabular}  

\vspace{0.5cm}

\begin{tabular}{|r|c|c|c|c|c|}
\hline
& \multicolumn{5}{|c|}{ CG + IC Iterations / Time (sec)} \\
\hline
$h^{-1}$ & $p=2$ & $p=3$ & $p=4$ & $p=5$ & $p=6$ \\
\hline
32 & 24 / \z0.44 & 18 / \z\z1.28 & 15 / \z\z3.61 & 12 / \z9.63 & 11 / \z24.57 \\
\hline
64 & 47 / \z5.19 & 35 / \z14.95 & 28 / \z37.33 & 24 / 94.08 & 20 / 222.09 \\
\hline
128 & 94 / 81.65 & 71 / 211.53 & 57 / 464.84 & * & * \\
\hline
\end{tabular}  
\caption{Revolved quarter of ring domain. Performance of CG preconditioned by the direct method (upper table), by ADI (middle table) and by Incomplete Cholesky (lower table).}
\label{tab:rev_ring}
\end{center}
\end{table}

Somewhat surprisingly, however, the CPU times show a stronger dependence on $p$ than in the 2D case, and the performance gap between the ADI and the FD approach is not as large as for the cube domain.
This is due to the cost of the residual computation  in the
CG iteration (a sparse matrix-vector product, costing $O(p^3 n^3)$
FLOPs). This step represents now a significant computational effort in the overall CG performance.
In fact, our numerical experience shows that the 3D FD method is
so efficient that the time spent in the preconditioning step is often negligible w.r.t. the time required for
the residual computation.
This effect is clearly shown in Table \ref{tab:percentage}, where we we report the percentage of time spent in the application of the preconditioner when compared with the overall time of CG, in the case of the revolved ring domain.
Interestingly, this percentage is almost constant w.r.t. $h$ up to the finest discretization level, corresponding to about 2 million degrees-of-freedom.

\begin{table}
\begin{center}
\footnotesize
\begin{tabular}{|r|c|c|c|c|c|}
\hline
$h^{-1}$ & $p=2$ & $p=3$ & $p=4$ & $p=5$ & $p=6$ \\
\hline
32 & 25.60 & 13.34 & 7.40 & 4.16 & 2.44 \\
\hline
64 & 22.69 & 11.26 & 5.84 & 3.32 & 1.88 \\
\hline
128 & 25.64 & 13.09 & 6.92 & * & * \\
\hline
\end{tabular}  
\caption{Percentage of time spent in the application of the 3D FD preconditioner with respect to the overall CG time. Revolved ring domain.}
\label{tab:percentage}
\end{center}
\end{table}


\subsection{Multi-patch experiments}

\begin{figure}
\begin{center}
\includegraphics[trim=0cm 0cm 0cm 0cm, clip=true, width=\textwidth]{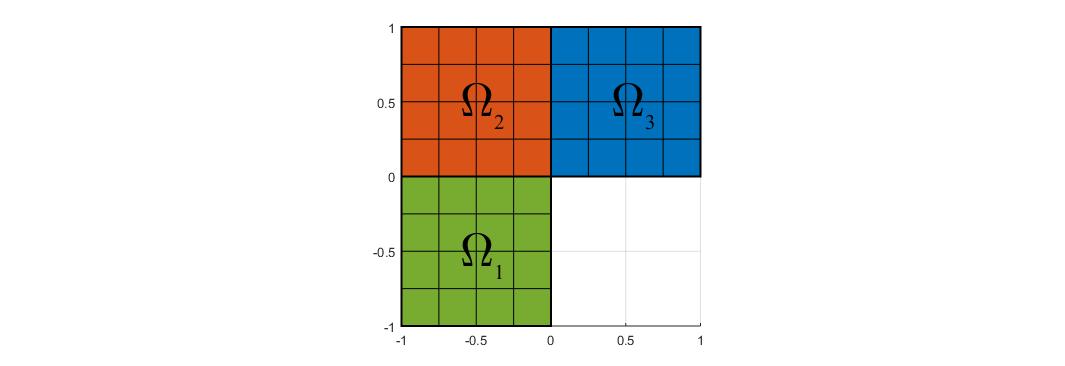}
\caption{L-shaped domain domain (3 patches)}
\label{fig:L}
\end{center}
\end{figure}

In this section we consider a multi-patch 2D problem. We consider the
L-shaped domain shown in Figure \ref{fig:L}, discretized with 3
patches and imposing $C^{0}$ continuity at the interfaces. We solve
this problem using the additive overlapping Schwarz preconditioner
described in Section \ref{sec:multi-patch}. Here $\Omega$ is split into two rectangular subdomains, which overlap on $\Omega_2$. Despite its apparent simplicity, this problem contains all the relevant ingredients to test the validity of our approach in the multi-patch case.
Note in particular that the Jacobian of the geometry mapping on the subdomains is not the identity matrix, due to stretching in either direction. This stretching could be easily incorporated into the preconditioner, but we avoid doing that in order to mimic the effect of nontrivial geometry.

We consider both the exact preconditioner, where the local systems are solved with {\sc Matlab}'s direct solver, and the inexact one \eqref{eq:Pias}, with $\widetilde{\A}_i = \p_i$, implemented by the FD solver. For completeness, we also consider the IC preconditioner, like in previous experiments. The results are shown in Table \ref{tab:L}. 
At the coarsest level and for low degree, the three approaches perform similarly, in terms of CPU time. However, as expected the $\p_{IAS}$ preconditioner scales much better than the others w.r.t. $h$ and $p$. In particular, the number of iterations is independent of these parameters. 

We also tested with $\p_{IAS}$, when however the matrices $\widetilde{\A}^{-1}_i$ represent a fixed number of preconditioned CG iterations. The results are comparable with the ones obtained for $\widetilde{\A}_i = \p_i$, although slightly worse, and hence they are not shown.

\begin{table}
\begin{center}
\footnotesize
\begin{tabular}{|r|c|c|c|c|c|}
\hline
& \multicolumn{5}{|c|}{ CG + $\p_{EAS}$ Iterations / Time (sec)} \\
\hline
$h^{-1}$ & $p=1$ & $p=2$ & $p=3$ & $p=4$ & $p=5$  \\
\hline
128 & 3 / \z\z0.88 & 2 / \z\z1.17 & 2 / \z\z2.51 & 2 / \z\z3.82 & 2 / \z\z\z6.14 \\
\hline
256 & 2 / \z\z2.98 & 2 / \z\z6.43 & 2 / \z16.44 & 2 / \z24.98 & 2 / \z\z41.27 \\
\hline
512 & 2 / \z17.42 & 2 / \z34.83 & 2 / 128.71 & 2 / 168.33 & 2 / \z334.95 \\
\hline
1024 & 2 / 110.55 & 2 / 266.03 & 2 / 998.72 & 2 / 2028.72 & 2 / 2530.75 \\
\hline
\end{tabular}  

\vspace{0.5cm}

\begin{tabular}{|r|c|c|c|c|c|}
\hline
& \multicolumn{5}{|c|}{ CG + $\p_{IAS}$ Iterations / Time (sec)} \\
\hline
$h^{-1}$ & $p=1$ & $p=2$ & $p=3$ & $p=4$ & $p=5$  \\
\hline
128 & 20 / \z0.79 & 20 / \z0.54 & 20 / \z0.64 & 19 / \z0.74 & 19 / \z0.86 \\
\hline
256 & 19 / \z1.46 & 20 / \z1.69 & 19 / \z1.91 & 19 / \z2.20 & 19 / \z2.55 \\
\hline
512 & 19 / \z7.32 & 19 / \z7.50 & 19 / \z8.11 & 19 / \z9.07 & 19 / 10.23 \\
\hline
1024 & 19 / 44.40 & 19 / 44.57 & 19 / 49.50 & 18 / 61.42 & 18 / 55.84 \\
\hline
\end{tabular}  

\vspace{0.5cm}

\begin{tabular}{|r|c|c|c|c|c|}
\hline
& \multicolumn{5}{|c|}{ CG + IC Iterations / Time (sec)} \\
\hline
$h^{-1}$ & $p=1$ & $p=2$ & $p=3$ & $p=4$ & $p=5$  \\
\hline
128 & \z144 / \z\z0.58 & \z94 / \z\z0.57 & \z69 / \z\z0.58  & \z56 / \z\z0.71 & \z46 / \z\z0.89 \\
\hline
256 & \z280 / \z\z3.76 & 180 / \z\z4.07 & 132 / \z\z4.65 & 106 / \z\z5.46 & \z87 / \z\z6.23 \\
\hline
512 & \z544 / \z31.07 & 348 / \z31.86 & 254 / \z36.10 & 203 / \z42.44 & 166 / \z48.50 \\
\hline
1024 & 1052 / 237.01 & 673 / 246.65 & 491 / 273.97 & 392 / 325.38 & 321 / 470.61 \\
\hline
\end{tabular} 
\caption{L-shaped domain.  Performance of CG preconditioned by $\p_{EAS}$ (upper table), by $\p_{IAS}$ (middle table) and by Incomplete Cholesky (lower table).}
\label{tab:L}
\end{center}
\end{table}

\section{Conclusions} \label{sec:conclusions}
In this work we have analyzed and tested the use of fast solvers for Sylvester-like equations as preconditioners for isogeometric discretizations. 


We considered here a Poisson problem on a single-patch domain, and we focused on the $k$-method, i.e., splines with maximal smoothness.
%
The considered preconditioner $\p$ is robust w.r.t. $h$ and $p$, and we have compared two popular methods for its application. We found that the FD direct solver, especially in 3D,  is by far more effective than the ADI iterative solver. Both approaches easily outperform a simple-minded Incomplete Cholesky preconditioner. 

Our conclusion is that the use of the FD method is,
likely, the best possible choice to compute the action of the operator
$\p^{-1}$.  
In fact, even if a more efficient solver was available, its employment would not necessary yield a relevant improvement in the overall performance, since in our  experiments the cost to apply the FD solver is already negligible w.r.t. the cost of a single matrix-vector product.
This is, then, a very promising  preconditioning stage in an iterative solver for isogeometric
discretizations. 
As we showed, this preconditioner can be easily combined with a domain decomposition strategy to solve multi-patch problems with conforming discretization.
In  a forthcoming paper we will further study the
role of the geometry parametrization on the performance of the
approaches based on Sylvester equation solvers, and propose possible strategies to improve it.

\section*{Acknowledgments}
The authors would like to thank  Valeria Simoncini  for fruitful discussions on the topics of the paper.
The authors were partially supported by the European Research Council
through the FP7 Ideas Consolodator Grant \emph{HIGEOM} n.616563. This support is
gratefully acknowledged.

\bibliography{biblio_sylvester_gian}

\begin{thebibliography}{10}

\bibitem{Ballani2013}
{\sc J.~Ballani and L.~Grasedyck}, {\em A projection method to solve linear
  systems in tensor format}, Numerical Linear Algebra with Applications, 20
  (2013), pp.~27--43.

\bibitem{Bartels1972}
{\sc R.~H. Bartels and G.~W. Stewart}, {\em Solution of the matrix equation
  {AX+ XB= C}}, Communications of the ACM, 15 (1972), pp.~820--826.

\bibitem{bazilevs2010isogeometric}
{\sc Y.~Bazilevs, C.~Michler, V.~M. Calo, and T.~J.~R. Hughes}, {\em
  Isogeometric variational multiscale modeling of wall-bounded turbulent flows
  with weakly enforced boundary conditions on unstretched meshes}, Computer
  Methods in Applied Mechanics and Engineering, 199 (2010), pp.~780--790.

\bibitem{acta-IGA}
{\sc L.~Beir{\~a}o~da Veiga, A.~Buffa, G.~Sangalli, and R.~V\'{a}zquez}, {\em
  Mathematical analysis of variational isogeometric methods}, Acta Numerica, 23
  (2014), pp.~157--287.

\bibitem{Veiga2012}
{\sc L.~Beir{\~a}o~da Veiga, D.~Cho, L.~F. Pavarino, and S.~Scacchi}, {\em
  Overlapping {S}chwarz methods for isogeometric analysis}, SIAM Journal on
  Numerical Analysis, 50 (2012), pp.~1394--1416.

\bibitem{BeiraodaVeiga2013}
\leavevmode\vrule height 2pt depth -1.6pt width 23pt, {\em {BDDC}
  preconditioners for isogeometric analysis}, Mathematical Models and Methods
  in Applied Sciences, 23 (2013), pp.~1099--1142.

\bibitem{Benner2013}
{\sc P.~Benner and T.~Breiten}, {\em Low rank methods for a class of
  generalized {L}yapunov equations and related issues}, Numerische Mathematik,
  124 (2013), pp.~441--470.

\bibitem{Benner2009}
{\sc P.~Benner, R.~C. Li, and N.~Truhar}, {\em On the {ADI} method for
  {S}ylvester equations}, Journal of Computational and Applied Mathematics, 233
  (2009), pp.~1035--1045.

\bibitem{bercovier2015overlapping}
{\sc M.~Bercovier and I.~Soloveichik}, {\em Overlapping non matching meshes
  domain decomposition method in isogeometric analysis}, arXiv preprint
  arXiv:1502.03756,  (2015).

\bibitem{Buffa2013}
{\sc A.~Buffa, H.~Harbrecht, A.~Kunoth, and G.~Sangalli}, {\em
  {BPX}-preconditioning for isogeometric analysis}, Computer Methods in Applied
  Mechanics and Engineering, 265 (2013), pp.~63--70.

\bibitem{Collier2013}
{\sc N.~Collier, L.~Dalcin, D.~Pardo, and V.~M. Calo}, {\em The cost of
  continuity: performance of iterative solvers on isogeometric finite
  elements}, SIAM Journal on Scientific Computing, 35 (2013), pp.~A767--A784.

\bibitem{Collier2012}
{\sc N.~Collier, D.~Pardo, L.~Dalcin, M.~Paszynski, and V.~M. Calo}, {\em The
  cost of continuity: a study of the performance of isogeometric finite
  elements using direct solvers}, Computer Methods in Applied Mechanics and
  Engineering, 213 (2012), pp.~353--361.

\bibitem{Cottrell2009}
{\sc J.~A. Cottrell, T.~J.~R. Hughes, and Y.~Bazilevs}, {\em Isogeometric
  analysis: toward integration of {CAD} and {FEA}}, John Wiley \& Sons, 2009.

\bibitem{cottrell2007studies}
{\sc J.~A. Cottrell, T.~J.~R. Hughes, and A.~Reali}, {\em Studies of refinement
  and continuity in isogeometric structural analysis}, Computer Methods in
  Applied Mechanics and Engineering, 196 (2007), pp.~4160--4183.

\bibitem{Damm2008}
{\sc T.~Damm}, {\em Direct methods and {ADI}-preconditioned {K}rylov subspace
  methods for generalized {L}yapunov equations}, Numerical Linear Algebra with
  Applications, 15 (2008), pp.~853--871.

\bibitem{DeFalco2011}
{\sc C.~De~Falco, A.~Reali, and R.~V{\'a}zquez}, {\em Geopdes: a research tool
  for isogeometric analysis of pdes}, Advances in Engineering Software, 42
  (2011), pp.~1020--1034.

\bibitem{Demmel1997}
{\sc J.~W. Demmel}, {\em Applied numerical linear algebra}, SIAM, 1997.

\bibitem{Deville2002}
{\sc M.~O. Deville, P.~F. Fischer, and E.~H. Mund}, {\em High-order methods for
  incompressible fluid flow}, Cambridge University Press, 2002.

\bibitem{Donatelli2015}
{\sc M.~Donatelli, C.~Garoni, C.~Manni, S.~Serra-Capizzano, and H.~Speleers},
  {\em Robust and optimal multi-iterative techniques for {I}g{A} {G}alerkin
  linear systems}, Computer Methods in Applied Mechanics and Engineering, 284
  (2015), pp.~230--264.

\bibitem{Douglas1962}
{\sc J.~Douglas}, {\em Alternating direction methods for three space
  variables}, Numerische Mathematik, 4 (1962), pp.~41--63.

\bibitem{Ellner1991}
{\sc N.~S. Ellner and E.~L. Wachspress}, {\em Alternating direction implicit
  iteration for systems with complex spectra}, SIAM journal on numerical
  analysis, 28 (1991), pp.~859--870.

\bibitem{Gahalaut2013}
{\sc K.~P.~S. Gahalaut, J.~K. Kraus, and S.~K. Tomar}, {\em Multigrid methods
  for isogeometric discretization}, Computer Methods in Applied Mechanics and
  Engineering, 253 (2013), pp.~413--425.

\bibitem{Gao2013}
{\sc L.~Gao}, {\em Kronecker Products on Preconditioning}, PhD thesis, King
  Abdullah University of Science and Technology, 2013.

\bibitem{Golub2012}
{\sc G.~H. Golub and C.~F. Van~Loan}, {\em Matrix computations}, Johns Hopkins
  University Press, 2012.

\bibitem{Grasedyck2004}
{\sc L.~Grasedyck}, {\em Existence and computation of low {K}ronecker-rank
  approximations for large linear systems of tensor product structure},
  Computing, 72 (2004), pp.~247--265.

\bibitem{Hofreither2015}
{\sc C.~Hofreither, S.~Takacs, and W.~Zulehner}, {\em A robust multigrid method
  for isogeometric analysis using boundary correction}, Tech. Report~33, NFN,
  2015.

\bibitem{Hughes2005}
{\sc T.~J.~R. Hughes, J.~A. Cottrell, and Y.~Bazilevs}, {\em Isogeometric
  analysis: {CAD}, finite elements, {NURBS}, exact geometry and mesh
  refinement}, Computer Methods in Applied Mechanics and Engineering, 194
  (2005), pp.~4135--4195.

\bibitem{Kleiss2012}
{\sc S.~K. Kleiss, C.~Pechstein, B.~J{\"u}ttler, and S.~Tomar}, {\em
  {IETI}--{I}sogeometric tearing and interconnecting}, Computer Methods in
  Applied Mechanics and Engineering, 247 (2012), pp.~201--215.

\bibitem{Kressner2014}
{\sc D.~Kressner, M.~Ple{\v{s}}inger, and C.~Tobler}, {\em A preconditioned
  low-rank {CG} method for parameter-dependent {L}yapunov matrix equations},
  Numerical Linear Algebra with Applications, 21 (2014), pp.~666--684.

\bibitem{Kressner2010}
{\sc D.~Kressner and C.~Tobler}, {\em Krylov subspace methods for linear
  systems with tensor product structure}, SIAM Journal on Matrix Analysis and
  Applications, 31 (2010), pp.~1688--1714.

\bibitem{Li2010}
{\sc B.~W. Li, S.~Tian, Y.~S. Sun, and Z.~M. Hu}, {\em Schur-decomposition for
  3{D} matrix equations and its application in solving radiative discrete
  ordinates equations discretized by {C}hebyshev collocation spectral method},
  Journal of Computational Physics, 229 (2010), pp.~1198--1212.

\bibitem{Lu1991}
{\sc A.~Lu and E.~L. Wachspress}, {\em Solution of {L}yapunov equations by
  alternating direction implicit iteration}, Computers \& Mathematics with
  Applications, 21 (1991), pp.~43--58.

\bibitem{Lynch1964}
{\sc R.~E. Lynch, J.~R: Rice, and D.~H. Thomas}, {\em Direct solution of
  partial difference equations by tensor product methods}, Numerische
  Mathematik, 6 (1964), pp.~185--199.

\bibitem{Peaceman1955}
{\sc D.~W. Peaceman and H.~H. Rachford}, {\em The numerical solution of
  parabolic and elliptic differential equations}, Journal of the Society for
  Industrial \& Applied Mathematics, 3 (1955), pp.~28--41.

\bibitem{Penzl2000}
{\sc T.~Penzl}, {\em A cyclic low-rank {S}mith method for large sparse
  {L}yapunov equations}, SIAM Journal on Scientific Computing, 21 (2000),
  pp.~1401--1418.

\bibitem{Simoncini2013}
{\sc V.~Simoncini}, {\em Computational methods for linear matrix equations}, to
  appear on SIAM Review,  (2013).

\bibitem{Sorber2014}
{\sc L.~Sorber, M.~Van~Barel, and L.~De~Lathauwer}, {\em Tensorlab v2. 0},
  Available online, URL: www. tensorlab. net,  (2014).

\bibitem{Takacs2015}
{\sc S.~Takacs and T.~Takacs}, {\em Approximation error estimates and inverse
  inequalities for {B}-splines of maximum smoothness}, arXiv preprint
  arXiv:1502.03733,  (2015).

\bibitem{Wachspress1962}
{\sc E.~L. Wachspress}, {\em Optimum alternating-direction-implicit iteration
  parameters for a model problem}, Journal of the Society for Industrial \&
  Applied Mathematics, 10 (1962), pp.~339--350.

\bibitem{Wachspress1963}
\leavevmode\vrule height 2pt depth -1.6pt width 23pt, {\em Extended application
  of alternating direction implicit iteration model problem theory}, Journal of
  the Society for Industrial \& Applied Mathematics, 11 (1963), pp.~994--1016.

\bibitem{Wachspress1994}
\leavevmode\vrule height 2pt depth -1.6pt width 23pt, {\em Three-variable
  alternating-direction-implicit iteration}, Computers \& Mathematics with
  Applications, 27 (1994), pp.~1--7.

\bibitem{Wachspress2013}
\leavevmode\vrule height 2pt depth -1.6pt width 23pt, {\em The ADI model
  problem}, Springer, 2013.

\end{thebibliography}

\end{document}